\documentclass{amsart}
\usepackage{graphicx}
\newtheorem{theorem}{Theorem}[section]
\newtheorem{lemma}[theorem]{Lemma}
\newtheorem{corollary}[theorem]{Corollary}

\theoremstyle{definition}
\newtheorem{definition}[theorem]{Definition}
\newtheorem{example}[theorem]{Example}

\newtheorem{claim}[theorem]{Claim}
\newtheorem{conjecture}[theorem]{Conjecture}

\theoremstyle{remark}
\newtheorem{remark}[theorem]{Remark}

\numberwithin{equation}{section}

\begin{document}

\title{On the Neuwirth conjecture for knots}

\author{Makoto Ozawa}
\address{Department of Natural Sciences, Faculty of Arts and Sciences, Komazawa University, 1-23-1 Komazawa, Setagaya-ku, Tokyo, 154-8525, Japan}
\email{w3c@komazawa-u.ac.jp}

\author{J. Hyam Rubinstein}
\address{Department of Mathematics and Statistics, The University of Melbourne
Parkville, Victoria 3010, Australia}
\email{rubin@ms.unimelb.edu.au}

\subjclass[2000]{Primary 57M25}



\keywords{Neuwirth conjecture, Even boundary slope conjecture, pre-essential surface, Murasugi sum, knot minor, pretzel surface, Montesinos knot, generalized arborescently alternating link, degree one map, normal surface, algorithm}

\begin{abstract}
Neuwirth asked if any non-trivial knot in the 3-sphere can be embedded in a closed surface so that the complement of the surface is a connected essential surface for the knot complement. In this paper, we examine some variations on this question and prove it for all knots up to 11 crossings except for two examples. We also establish the conjecture for all Montesinos knots and for all generalized arborescently alternating knots. For knot exteriors containing closed incompressible surfaces satisfying a simple homological condition, we establish that the knots satisfy the Neuwirth conjecture. If there is a proper degree one map from knot $K$ to knot $K'$ and $K'$ satisfies the Neuwirth conjecture then we prove the same is true for knot $K$. Algorithms are given to decide if a knot satisfies the various versions of the Neuwirth conjecture and also the related conjectures about whether all non-trivial knots have essential surfaces at integer boundary slopes.
\end{abstract}

\maketitle
\tableofcontents

\section{Conjectures}

In 1964, Neuwirth conjectured in \cite{N}, \cite{N2} that the fundamental group of the complement of a non-trivial knot in the 3-sphere is a non-trivial free product with amalgamation, and the amalgamating subgroup is free.
In 1984, this conjecture was solved by Culler--Shalen, by realising such an algebraic splitting via a suitable properly embedded surface in the exterior of the knot. 

\begin{theorem}[Weak Neuwirth Conjecture, \cite{CS}]
For any non-trivial knot $K$, there exists a properly embedded separating, orientable, incompressible and boundary incompressible surface in the exterior $E(K)$.
\end{theorem}

However, a geometrical conjecture which is an original source of the Weak Neuwirth Conjecture has not been solved.
The following Neuwirth Conjecture asserts that any non-trivial knot can be embedded in a closed surface, similarly to the way a torus knot can be embedded in an unknotted torus.

\begin{conjecture}[Neuwirth Conjecture, \cite{N}]
For any non-trivial knot $K$, there exists a closed surface $F$ containing $K$ such that $F\cap E(K)$ is connected, incompressible and boundary incompressible.
\end{conjecture}

The following knot classes are known to satisfy the Neuwirth Conjecture.

\begin{itemize}
\item Torus knots and cable knots
\item 2-bridge knots
\item Alternating knots ({\cite[Theorem 9.8]{A}}, {\cite[Proposition 2.3]{MT}})
\item Generalized alternating knots ({\cite[Theorem 2]{O}})
\item Non-positive $+$-adequate knots (\cite[Corollary 2.2]{O2})
\item Crosscap number 2 knots (\cite[Theorem 6]{IOT})
\item Tunnel number 2 knots which can be non-separatingly embedded in a genus two Heegaard surface (\cite[Lemma 2.3]{O6}, \cite[Lemma 1]{O4})
\item Knots with accidental surfaces with non-separating accidental peripherals (\cite[Theorem 2]{IO})
\end{itemize}

We will show in Corollary \ref{simple}, that if a knot $K$ satisfies the Neuwirth Conjecture then so does any satellite knot obtained from $K$ and any composite knot obtained by summing a knot with $K$. This implies that to prove the Neuwirth Conjecture, it suffices to consider only 
simple knots. 

Almost all known examples of knots satisfying the Neuwirth Conjecture, are obtained by taking boundaries of regular neighbourhoods of algebraically incompressible and boundary incompressible non-orientable spanning surfaces, except for torus knots and the last two classes in the above list.
Therefore, the following Strong Neuwirth Conjecture is plausible.

\begin{conjecture}[Strong Neuwirth Conjecture, {\cite[Question 5]{IOT}}]\label{strong}
For any prime non-torus knot $K$, there exists a non-orientable spanning surface $F$ for $K$ such that $F\cap E(K)$ is algebraically incompressible and boundary incompressible.
\end{conjecture}

It seems to be unknown whether Conjecture \ref{strong} holds even if the condition weakened.

\begin{conjecture}[Weakly Strong Neuwirth Conjecture]
For any prime non-torus knot $K$, there exists a non-orientable spanning surface $F$ for $K$ such that $F\cap E(K)$ is geometrically incompressible and boundary incompressible.
\end{conjecture}

\begin{remark}
It can be observed that a composite knot bounds an algebraically (geometrically) incompressible and boundary incompressible non-orientable spanning surface if and only if at least one of the factor knots also bounds an algebraically (geometrically) incompressible and boundary incompressible non-orientable spanning surface.
\end{remark}

The existence of an algebraically incompressible and boundary incompressible non-orientable spanning surface in Conjecture \ref{strong} implies the following Conjecture \ref{strong even}.
It seems to be unknown whether Conjecture \ref{strong even} holds even if the boundary slope is non-integer.

\begin{conjecture}[Even Boundary Slope Conjecture]\label{even}
For any prime non-torus knot $K$, there is a properly embedded orientable incompressible and boundary incompressible surface, which is not a Seifert surface, in the exterior $E(K)$ with boundary slope an even rational number. 
\end{conjecture}

\begin{conjecture}[Strong Even Boundary Slope Conjecture]\label{strong even}
For any prime non-torus knot $K$, there is a properly embedded orientable incompressible and boundary incompressible surface, which is not a Seifert surface, in the exterior $E(K)$ with boundary slope an even integer.
\end{conjecture}

We will show in Corollary \ref{simple2}, that if a knot $K$ satisfies the (strong) even boundary slope conjecture then so does any satellite knot obtained from $K$ and any composite knot obtained by summing a knot with $K$. This implies that to prove the (strong) even boundary slope conjecture, it suffices to consider only simple knots. 


\begin{remark}
Miyazaki (\cite{M}) showed that for any integer $m \ge 0$, there is a hyperbolic knot which has $m+1$ accidental surfaces with accidental slopes $0, 1, \ldots ,m$.
Tsutsumi (\cite{T2}) showed that for any finite set of even integers $\{a_1,\ldots,a_n\}$ and for any closed connected 3-manifold $M$, there exists an excellent knot in $M$ which bounds excellent non-orientable spanning surfaces $F_1,\ldots,F_n$ such that the boundary slope of $F_i$ is $a_i$.
These two constructions show that there exists a hyperbolic knot with finitely many Neuwirth surfaces at finitely many integer boundary slopes.
\end{remark}

\section{Known definitions and results}

\subsection{Geometrically incompressible and algebraically incompressible surfaces}
We review the definition of essential surfaces in both the geometric and algebraic senses.

Let $M$ be an orientable compact 3-manifold, $F$ a compact surface properly embedded in $M$, possibly with boundary, except for a 2-sphere or disk, and let $i$ denote the inclusion map $F\to M$.
We say that $F$ is {\em algebraically incompressible} if the induced map $i_*:\pi_1(F)\to \pi_1(M)$ is injective, and that $F$ is {\em algebraically boundary incompressible} if the induced map $i_*:\pi_1(F,\partial F)\to \pi_1(M,\partial M)$ is injective for every choice of two base points in $\partial F$.

A disk $D$ embedded in $M$ is a {\em compressing disk} for $F$ if $D\cap F=\partial D$ and $\partial D$ is an essential loop in $F$.
A disk $D$ embedded in $M$ is a {\em boundary compressing disk} for $F$ if $D\cap F\subset \partial D$ is an essential arc in $F$ and $D\cap \partial M=\partial D-\rm{int}(D\cap F)$.
We say that $F$ is {\em geometrically incompressible} (resp. {\em geometrically boundary incompressible}) if there exists no compressing disk (resp. boundary compressing disk) for $F$.

\subsection{Murasugi sum and knot minors}

Let $F$ be a spanning surface for a knot or link $K$.
Suppose that there exists a 2-sphere $S$ decomposing $S^3$ into two 3-balls $B_1,B_2$ such that $S$ intersects $K$ transversely and $F\cap S$ consists of a disk.
Put $F_i=F\cap B_i$ for $i=1,2$.
Then we say that $F$ has a {\em Murasugi decomposition} into $F_1$ and $F_2$ and we denote by $F=F_1*F_2$.
Conversely, we say that $F$ is obtained from $F_1$ and $F_2$ by a {\em Murasugi sum} along a disk $F\cap S$.
We say that a Murasugi sum (or Murasugi decomposition) is {\em plumbing} (or {\em deplumbing}) if $S$ intersects $K$ in $4$ points.

The Murasugi sum is a natural geometric operation.
In fact, Gabai proved that geometrically incompressibility for Seifert surfaces is preserved under Murasugi sums, and the first author showed that algebraically incompressibility for spanning surfaces is also preserved under Murasugi sums.

\begin{theorem}[{\cite[Theorem 1]{G}}, {\cite[Lemma 3.4]{O2}}]\label{Murasugi sum}
If $F_1$ and $F_2$ are algebraically incompressible and boundary incompressible, then $F=F_1*F_2$ is also algebraically incompressible and boundary incompressible.
\end{theorem}

We say that a knot or link $K$ has a {\em minor} $K_i$ if there exists algebraically incompressible and boundary incompressible spanning surfaces $F_i$ for $K_i$ such that $F=F_1*F_2$.
In the proof of Theorem \ref{Montesinos}, this concept plays a central role and it turns out that any Montesinos knot except for torus knots has a pretzel knot or link minor.

\subsection{State surfaces}

In the following, we review the definitions and results on state surfaces, which are introduced in \cite{O2}.

Let $K$ be a knot or link in the 3-sphere $S^3$ and $D$ a connected diagram of $K$ on the 2-sphere $S^2$ which
 separates $S^3$ into two 3-balls, say $B_+, B_-$.
Let $\mathcal{C}=\{c_1,\ldots,c_n\}$ be the set of crossings of $D$.
A map $\sigma : \mathcal{C}\to \{+,-\}$ is called a {\em state} for $D$.
For each crossing $c_i\in \mathcal{C}$, we take a $+$-smoothing or $-$-smoothing according to $\sigma(c_i)=+$ or $-$.
See Figure \ref{resolution}.
Then, we have a collection of loops $l_1,\ldots, l_m$ on $S^2$ and call these {\em state loops}.
Let $\mathcal{L}_{\sigma} = \{l_1,\ldots, l_m\}$ be the set of state loops.

\begin{figure}[htbp]
	\begin{center}
		\includegraphics[trim=0mm 0mm 0mm 0mm, width=.6\linewidth]{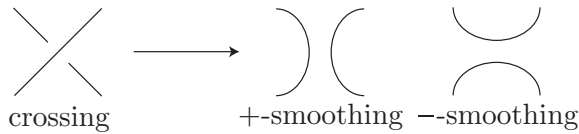}
	\end{center}
	\caption{Two smoothings of a crossing}
	\label{resolution}
\end{figure}

Each state loop $l_i$ bounds a unique disk $d_i$ in $B_-$, and we may assume that these disks are mutually disjoint.
For each crossing $c_j$ and state loops $l_i,l_k$ whose subarcs replaced $c_j$ by a $\sigma(c_j)$-smoothing, we attach a half twisted band $b_j$ to $d_i,d_k$ so that $c_j$ is recovered .
See Figure \ref{recover} for $\sigma(c_j)=+$.
In this way, we obtain a spanning surface which consists of disks $d_1,\ldots,d_m$ and half twisted bands $b_1,\ldots,b_n$ and call this a {\em $\sigma$-state surface} $F_{\sigma}$.

\begin{remark}\label{option}
Although we chose a disk $d_i$ in $B_-$ for each state loop $l_i$, we note that there are two options to choose a disk in $B_+$ or $B_-$ for each state loop which is not innermost in the 2-sphere $S^2$.
Therefore, in general, there are many state surfaces for a given state.
\end{remark}

\begin{figure}[htbp]
	\begin{center}
		\includegraphics[trim=0mm 0mm 0mm 0mm, width=.6\linewidth]{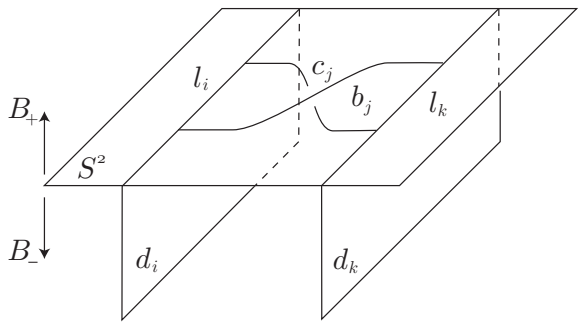}
	\end{center}
	\caption{Recovering a crossing by a half twisted band}
	\label{recover}
\end{figure}

We construct a graph $G_{\sigma}$ with signs on edges from $F_{\sigma}$ by regarding a disk $d_i$ as a vertex $v_i$ and a band $b_j$ as an edge $e_j$ which has the same sign $\sigma(c_j)$.
We call the graph $G_{\sigma}$ a {\em $\sigma$-state graph}.
In general, a graph is called a {\em block} if it is connected and has no cut vertex.
It is known that any graph has a unique decomposition into maximal blocks.
Following \cite{LT} and \cite{C1}, we say that a diagram $D$ is {\em $\sigma$-adequate} if $G_{\sigma}$ has no loops,
 and that $D$ is {\em $\sigma$-homogeneous} if in each block of $G_{\sigma}$, all edges have the same sign.

\begin{example}
Let $D$ be a diagram of the figure eight knot which has $4$ crossings $c_1,c_2,c_3,c_4$ as in Figure \ref{trefoil}.
To make a $\sigma$-state surface, let $\sigma(c_1)=\sigma(c_2)=-$ and $\sigma(c_3)=\sigma(c_4)=+$ for example.
Since the $\sigma$-state graph $G_{\sigma}$ has no loops and all edges in each block have the same sign as in Figure \ref{figure-8}, $D$ is $\sigma$-adequate and $\sigma$-homogeneous.
Moreover, the block decomposition of $G_{\sigma}$ corresponds to a Murasugi decomposition of $F_{\sigma}$. See Figure \ref{decomposition_fig}.
\end{example}

\begin{figure}[htbp]
	\begin{center}
		\includegraphics[trim=0mm 0mm 0mm 0mm, width=.9\linewidth]{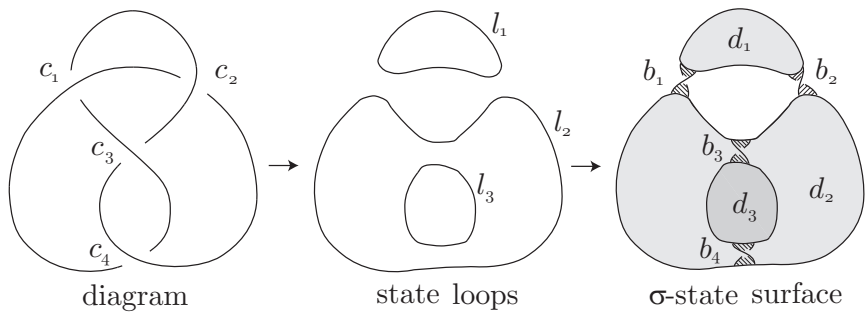}
	\end{center}
	\caption{An example of making a $\sigma$-state surface}
	\label{trefoil}
\end{figure}

\begin{figure}[htbp]
	\begin{center}
		\includegraphics[trim=0mm 0mm 0mm 0mm, width=.4\linewidth]{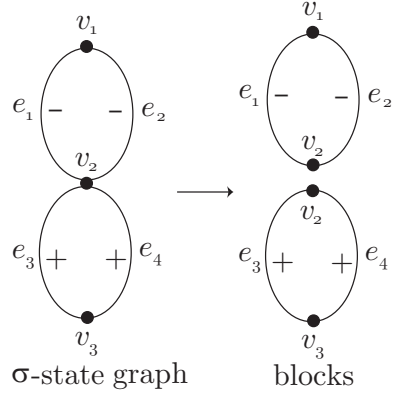}
	\end{center}
	\caption{The corresponding $\sigma$-state graph and its block decomposition}
	\label{figure-8}
\end{figure}

\begin{figure}[htbp]
	\begin{center}
		\includegraphics[trim=0mm 0mm 0mm 0mm, width=.6\linewidth]{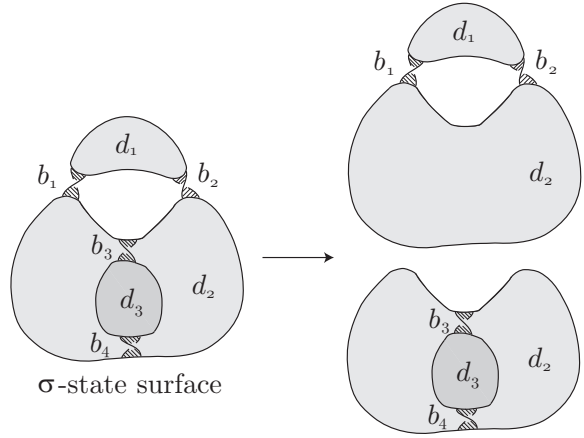}
	\end{center}
	\caption{The corresponding Murasugi decomposition}
	\label{decomposition_fig}
\end{figure}

\begin{example}
A diagram $D$ with an orientation is said to be {\em positive} if all crossings have a positive sign.
For any positive diagram $D$, there exists a state $\sigma$ such that $D$ is $\sigma$-adequate and $\sigma$-homogeneous.
Indeed, we can take $\sigma$ so that $\sigma(c_j)=+$ for all $c_j$, namely, the {\em positive state} $\sigma_+$.
Also we can take $\sigma$ so that it yields a canonical Seifert surface $F_{\sigma}$, namely, the {\em Seifert state} $\vec{\sigma}$.
Note that these states $\sigma_+$ and $\vec{\sigma}$ coincide only on a positive diagram.
\end{example}

\begin{example}
We say that a diagram $D$ is {\em $+$-adequate} (or {\em $-$-adequate}) \cite{LT} if $D$ is $\sigma$-adequate for the positive state $\sigma_+$ (or the negative state $\sigma_-$).
Note that $D$ is automatically $\sigma_{\pm}$-homogeneous since $\sigma_{\pm}(c_j)=\pm$ for all $j$.
Furthermore, we say that a diagram $D$ is {\em adequate} \cite{T} if $D$ is both $+$-adequate and $-$-adequate.
\end{example}

The Hasse diagram of various classes of knots and links is illustrated in Figure \ref{hasse}.
{\em Algebraically alternating} knots and links are defined in \cite{O3} so that they include both alternating and algebraic knots and links, and some results on closed incompressible surfaces are obtained.

\begin{figure}[htbp]
	\begin{center}
		\includegraphics[trim=0mm 0mm 0mm 0mm, width=.7\linewidth]{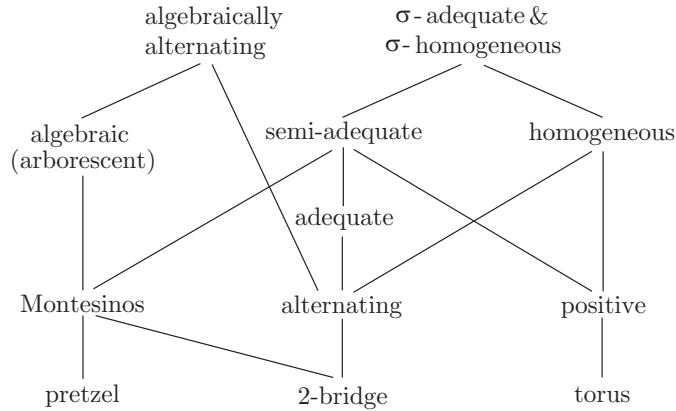}
	\end{center}
	\caption{The Hasse diagram for the set of knot diagrams partially ordered by inclusion}
	\label{hasse}
\end{figure}

\begin{theorem}[\cite{O2}]\label{state}
If a diagram is both $\sigma$-adequate and $\sigma$-homogeneous for some state $\sigma$, then the $\sigma$-state surface is algebraically incompressible and boundary incompressible.
\end{theorem}

\begin{corollary}[\cite{O2}]
If a diagram is $\sigma$-adequate and $\sigma$-homogeneous for a state $\sigma$ except for the Seifert state $\vec{\sigma}$, then the knot satisfies the strong Neuwirth conjecture.
In particular, adequate knots satisfy the strong Neuwirth conjecture.
\end{corollary}


Theorem \ref{state} is obtained from Theorem \ref{Murasugi sum} and the next Theorem \ref{checkerboard}.
Indeed, a knot satisfying the condition of Theorem \ref{state} has an alternating knot or link minor.

\begin{theorem}[{\cite[Theorem 9.8]{A}}, {\cite[Proposition 2.3]{MT}}, {\cite[Theorem 2]{O}}]\label{checkerboard}
If a diagram is reduced and alternating, then both of the checkerboard surfaces are algebraically incompressible and boundary incompressible.
\end{theorem}

\subsection{Rational tangles and Montesinos knots}

We recall that a {\em rational tangle} is a $2$-tangle that is homeomorphic to the trivial $2$-tangle as a map of pairs consisting of the $3$-ball and two arcs.
Usually, a rational tangle is inductively drawn as in Figure \ref{rational} which consists of $a_i$-twists.
We say that such a rational tangle diagram is {\em standard} if all $a_i$ are positive or negative, namely it is alternating.
It is known that any rational tangle has a standard form.

The {\em fraction} (or {\em slope}) of a rational tangle $(a_1,a_2,\ldots,a_n)$ is defined as the number given by the continued fraction $[a_n,a_{n-1},\ldots,a_1]$, where $a_1,\ldots a_{n-1}$ are non-zero integers and $a_n$ is an integer.
Conway (\cite{C}) proved that the fraction is well-defined and completely determines the rational tangle up to tangle equivalence.

\begin{figure}[htbp]
	\begin{center}
	\includegraphics[trim=0mm 0mm 0mm 0mm, width=.4\linewidth]{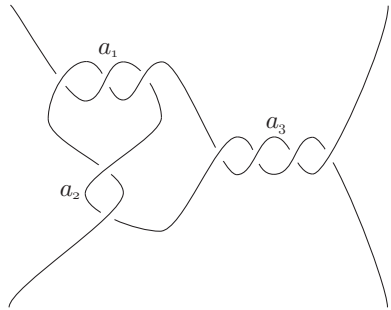}
	\end{center}
	\caption{A rational tangle $(3,2,4)$ with the fraction $[4,2,3]=31/7$}
	\label{rational}
\end{figure}

A knot or link $K$ is called {\em Montesinos} if $K$ has a form which is obtained by summing $n$ rational tangles $T_i$ with the slope $r_i$ as in Figure \ref{Montesinos figure}, and we denote the result as $K=M(r_1,\ldots,r_n)$.
A Montesinos knot or link $K=M(r_1,\ldots,r_n)$ is {\em pretzel} if all numerators of $r_i$ are $\pm 1$, namely $r_i=\pm 1/p_i$, and then we denote $K=P(p_1,\ldots,p_n)$.

\begin{figure}[htbp]
	\begin{center}
	\includegraphics[trim=0mm 0mm 0mm 0mm, width=.4\linewidth]{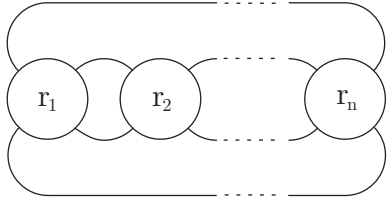}
	\end{center}
	\caption{A Montesinos knot or link $M(r_1,\ldots,r_n)$}
	\label{Montesinos figure}
\end{figure}

Similarly, a tangle $T$ is called {\em Montesinos} if $T$ has a form which is obtained by summing $n$ rational tangles $T_i$ with the slope $r_i$ as in Figure \ref{Montesinos tangle figure}, and we denote the result as $T=MT(r_1,\ldots,r_n)$.

\begin{figure}[htbp]
	\begin{center}
	\includegraphics[trim=0mm 0mm 0mm 0mm, width=.4\linewidth]{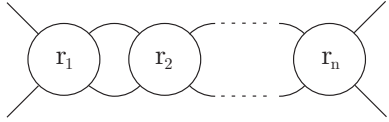}
	\end{center}
	\caption{A Montesinos tangle $MT(r_1,\ldots,r_n)$}
	\label{Montesinos tangle figure}
\end{figure}

\begin{lemma}\label{deform}
Let $MT(-r_1,r_2)$ be a Montesinos tangle with two rational tangles with slopes $-r_1, r_2$, where $0<r_i<1$ for $i=1,2$.
Then $MT(-r_1,r_2)$ can be deformed into a Montesinos tangle $MT(1-r_1,r_2-1)$.
\end{lemma}

\begin{proof}
See Figure \ref{deformation}.

\begin{figure}[htbp]
	\begin{center}
	\includegraphics[trim=0mm 0mm 0mm 0mm, width=.6\linewidth]{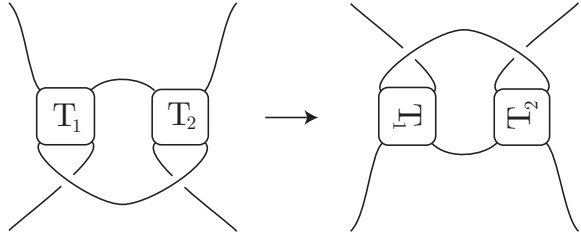}
	\end{center}
	\caption{A deformation from $T(-r_1,r_2)$ to $T(1-r_1,r_2-1)$}
	\label{deformation}
\end{figure}

We remark that the slope of $T_1$ is $-r_1/(1-r_1)$ and that of $T_2$ is $r_2/(1-r_2)$.
The slope $\phi(T)$ of a rational tangle $T$ can be calculated by the formula $\phi(T_1+T_2)=\phi(T_1)+\phi(T_2)$ and $\phi(T_1)\phi(T_1^*)=-1$, where $+$ denotes the tangle sum and $T^*$ denotes the rotation of $T$.
See \cite[Theorem 2.4]{O3} for example.
\end{proof}

\section{New definitions and results}

\subsection{Pre-essential surfaces}

\begin{definition}
Suppose $(S^3,K)$ is a knot and let $E(K)$ be the exterior of $K$ in $S^3$.

A {\em Neuwirth surface} $F$ for $K$ is a compact, orientable, geometrically incompressible and boundary incompressible surface properly embedded in $E(K)$ such that the number of boundary components of $F$ is two and the boundary slope of $F$ is an integer.

A {\em pre-essential surface} $S$ for $K$ is a compact non-orientable surface properly embedded in $E(K)$ with the property that the boundary $\partial N(S)$ of a regular neighbourhood of $S$ can be compressed to a properly embedded incompressible and boundary incompressible surface $F$ for $K$. Note that the boundary slope of the resulting surface $F$ is the same as the boundary slope of the pre-essential surface $S$.

A {\em quasi-spanning surface} is a compact non-orientable surface properly embedded in $E(K)$ with a single boundary component.
We note that any even rational boundary slope has such a quasi-spanning surface at this slope.
If a quasi-spanning pre-essential surface $S$ has an integer boundary slope, namely it is spanning pre-essential, then $\partial N(S)$ compresses to a Neuwirth surface.
Examples for spanning pre-essential surfaces are given in Example \ref{10_128} and \ref{10_139}.
\end{definition}

\begin{remark}
A pre-essential surface must have an odd number of boundary curves and their boundary slope is an even rational number. The reason is that if we attach an annulus along the boundary torus to two adjacent boundary curves of a pre-essential surface, then we can obtain a new non-orientable surface with two fewer boundary components. Since there are no closed embedded non orientable surfaces in $S^3$, it immediately follows that there must be an odd number of boundary curves. Reducing to one boundary curve, we see that the boundary slope must be even since it is zero in homology with coefficients in $\Bbb{Z}_2$. 
\end{remark}

A convenient criterion for a surface to be pre-essential is given by the following result. 

\begin{lemma}\label{pre-essential}
Suppose $(S^3,K)$ is a knot. A compact properly embedded non-orientable surface $S$ for $K$ is pre-essential if $S$ is algebraically boundary incompressible.
\end{lemma}

\begin{proof}
Assume that $S$ is a compact properly embedded non-orientable surface in $E(K)$. Suppose that $\partial N(S)$ does not compress to a properly embedded orientable incompressible and boundary incompressible surface $F$ for $K$. Then it is easy to deduce that after a series of compressions of $\partial N(S)$, a boundary parallel annulus is one of the components of the resulting surface. Consequently there is an essential arc $\lambda$ on such an annulus which is homotopic into $\partial E(K)$ keeping its ends fixed. But then $\lambda$ is also an essential arc on $\partial N(S)$ which is homotopic into $\partial E(K)$ implying that $S$ is not algebraically boundary incompressible. So this completes the proof. 
\end{proof}

We next show that if a knot has a pre-essential surface then it has a geometrically incompressible and boundary incompressible pre-essential surface. 

\begin{lemma}\label{geometrically incompressible}
Suppose that a knot $(S^3,K)$ has a pre-essential surface $S$. Then the knot admits a geometrically incompressible and geometrically boundary incompressible pre-essential surface $S_0$ with the same boundary slope as $S$. Moreover $S_0$ can be chosen to be algebraically boundary incompressible. 
\end{lemma}

\begin{proof}
Let $F$ be the properly embedded, orientable, incompressible and boundary incompressible surface obtained by compressing $\partial N(S)$. Suppose that compressing disks $D_1, D_2, \ldots, D_k$ are used to transform $\partial N(S)$ to $F$. Let $F_1,F_2, \ldots, F_k =F$ be the sequence of surfaces obtained by these compressions. (Any closed components of these surfaces can be discarded). Clearly $D_1$ must be on the side of $\partial N(S)$ away from $S$. However after several compressions, we may reach a compression disk $D_j$ for $F_{j-1}$ which is on the same side as $S$ and in fact which may meet $S$. However using a sequence of innermost disks of intersection of $D_j$ with $S$, we can geometrically compress $S$ to a new compact properly embedded non-orientable surface $S'$ which is disjoint from $F_{j-1}$ \emph{and} the compressing disk $D_j$ and hence also to $F_j$. So by repeatedly changing $S$ by compressions, we eventually obtain a new compact properly embedded non-orientable surface $S^*$ which is disjoint from $F$,

 Next, if we geometrically compress $S^*$, then this can be achieved in the complement of $F$, since any compressing disk for $S^*$ can be isotoped off of $F$. But then clearly $S^*$ cannot be completely compressible and so after a finite number of steps we end up with a non-orientable geometrically incompressible surface $S_0$ with the same boundary curve as $S$ and which is disjoint from $F$. But any algebraic boundary compression of $S_0$ would have to intersect $F$ and clearly an innermost intersection arc would give a non trivial boundary compression of $F$, contrary to our assumption that $F$ is properly embedded, orientable incompressible and boundary incompressible. This completes the proof that $S_0$ is geometrically incompressible and geometrically boundary incompressible, since it is algebraically boundary incompressible. Finally since $S_0$ is also algebraically boundary incompressible, it is pre-essential by our previous Lemma \ref{pre-essential}. 
\end{proof}

Putting together Lemmas \ref{pre-essential} and \ref{geometrically incompressible}, we have the following theorem.

\begin{theorem}
Let $(S^3,K)$ be a knot.
Then the following conditions are equivalent.
\begin{enumerate}
\item There exists a pre-essential surface $S$ with boundary slope $r$.
\item There exists an algebraically boundary incompressible non-orientable surface $S_0$ with boundary slope $r$.
\end{enumerate}
Moreover we can also assume that $S_0$ is geometrically incompressible (and geometrically boundary incompressible).
\end{theorem}

The following results Theorem \ref{homology} and Theorem \ref{large} give criteria producing an essential surface at an even boundary slope and the Neuwirth conjecture respectively, using a homological argument. We first give a key lemma required for both proofs. 

\begin{lemma}\label{normal}
Suppose that $(S^3,K)$ is a non-trivial knot and $X$ is a connected $2$-complex embedded in $E(K)$ with the properties that $X \cap \partial E(K)={\mathcal C}$ is a non-empty collection of disjoint simple closed curves, all at a non-zero rational even boundary slope and the mapping induced by inclusion $H_1(X,\Bbb{Z}_2) \to H_1(E(K),\Bbb{Z}_2)$ has image zero. Then there is a properly embedded compact connected non-orientable surface $F$ which is disjoint from $X$ and meets $\partial E(K)$ at a collection of simple closed curves $\mathcal C^\prime$ with the same boundary slope as the loops of $\mathcal C$. Moreover if $X$ is non-separating, then the number of curves in $\mathcal C^\prime$ is at most the number in $\mathcal C$. If $X$ is separating, the number of curves in $\mathcal C^\prime$ is at most half the number in $\mathcal C$.
\end{lemma}

\begin{proof}

See \cite{JS} for a discussion of basic concepts in normal surface theory.

We use a method from \cite{JRT}. Choose a triangulation $\mathcal T$ of $E(K)$ so that $X$ is a subcomplex.  Then the vertices and edges of $X$ are in the one-skeleton  $\mathcal T^1$ of $\mathcal T$. We want to build a maximal tree for $\mathcal T^1$  using Prim's (\cite{P}) and Kruskal's (\cite{K}) greedy algorithms. Note such a maximal tree is just a tree containing all the vertices of $\mathcal T$. Prim's algorithm takes any tree $T_0$ in $\mathcal T^1$ and adds an edge $E$ from $T_0$ to a vertex not in $T_0$ to form a new tree $T_1$. Continue until all vertices are reached. Kruskal's algorithm is similar, except that one can start with a forest and then add either an edge to one tree of the forest, or an edge connecting two trees of the forest together. 

Start with all except one of the edges of the loops of $X \cap \partial E(K)={\mathcal C}$. Using Kruskal's algorithm, we can enlarge this set of edges to a maximal tree $T_X$ in the one-skeleton of $X$. Next extend this to a maximal tree $T^*$ for the one skeleton of $X \cup \partial E(K)$ using Prim . It is easy to see that $T^* \cap \partial E(K)$ is a maximal forest  $\mathcal F$ for the edges of $\partial E(K)$. Note that each tree in the forest $\mathcal F$ contains all except one of the edges of a loop of $\mathcal C$. Finally extend this tree $T^*$ to a maximal tree $T$ for $\mathcal T^1$ again using Prim. 

Now each edge $E$ of ${\mathcal T} \setminus T$ can be labelled by either $0$ or $1$ depending on whether the cycle defined by connecting the ends of $E$ along a path in $T$ is $0$ or $1$ in $H_1(E(K),\Bbb{Z}_2) = \Bbb{Z}_2$. We also label all the edges in $T$ by $0$. It is easy to see that every face of $\mathcal T$ is then labelled either by all edges $0$ or two edges $1$ and one edge $0$ since the sum of the labels must be $0$ in $\Bbb{Z}_2$. Consequently each tetrahedron has all edges labelled $0$, or three edges at a common vertex labelled $1$ and the other three edges in a common face labelled $0$, or two opposite edges labelled $0$ and the other four edges labelled $1$. In the first case we choose the empty surface inside the tetrahedron, in the second a triangular normal disk with one corner on each edge labelled $1$ and in the third case a quadrilateral normal disk with one corner on each edge labelled $1$. See Figure 11. Doing this for every tetrahedron in $\mathcal T$ gives a properly embedded normal surface $F$. 

\begin{figure}[htbp]
	\begin{center}
	\includegraphics[trim=0mm 0mm 0mm 0mm, width=.6\linewidth]{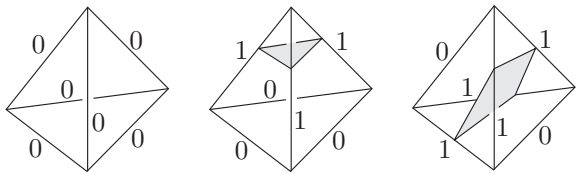}
	\end{center}
	\caption{Labeling edges of a triangulation $\mathcal T$}
	\label{labeling}
\end{figure}

We claim that $F$ is connected, disjoint from $X$, non-orientable and has a family of boundary curves $\mathcal C^\prime$ satisfying the conditions in the theorem. This will complete the proof. First of all, note that any component of $F$ is non-separating, since by definition, an edge containing a corner of a normal disk of $F$ meets $F$ in one point and the ends can be connected along $T$ in the complement of $F$. Hence we see that $F$ cannot have any closed components since there are no non-separating surfaces in $S^3$. Next, $F \cap X = \emptyset$. For by assumption, the map $H_1(X,\Bbb{Z}_2) \to H_1(E(K),\Bbb{Z}_2)$ has image zero. Hence any edge $E$ of ${\mathcal T} \setminus T$ which is in $X$ has label $0$, since $T \cap X$ consists of a maximal tree in $X$ and all these edges are also labelled $0$. But then all the normal disks of $F$ are disjoint from all the triangular faces of the subcomplex $X$. It now follows immediately also that $F$ is disjoint from any tetrahedra in $X$ and so $F \cap X = \emptyset$. 

The next step is to show that $F$ has a collection of boundary curves $\mathcal C^\prime$ parallel to the loops of $\mathcal C$. Once this is established, it follows that the components of $F$ must be non-orientable, because the boundary slope of the loops of $\mathcal C$ is not a longitude and the components of $F$ are non-separating. The labels of the edges of $\mathcal C$ are all $0$ since all these edges are in $X$. There must be edges labelled $1$ in $\partial E(K)$ since $H_1(\partial E(K), \Bbb{Z}_2) \to H_1(E(K), \Bbb{Z}_2)$ is onto. Hence we see that 
$\partial F \ne \emptyset$, each component $C$ of $\partial F$ is non separating and disjoint from the loops of $X \cap \partial E(K)$ by the previous argument. Consequently the components of $F$ are non-separating and there can only be one, since disjoint components could be connected by annuli in $\partial E(K)$ producing a closed non-orientable surface in $S^3$, a contradiction. 

Finally we want to bound the number of curves in $\partial F = {\mathcal C^\prime}$. Recall that the maximal tree $T$ contains a forest of edges $\mathcal F$ in $\partial E(K)$ with the property that each tree in $\mathcal F$ contains all the edges except one of a loop $C$ of $\mathcal C$. We claim this implies that there cannot be two parallel edges of $\partial F$ in any annulus component of $\partial E(K) \setminus {\mathcal C}$. Deducing the bound on the number of curves in $\partial F$ is then easy. 

Suppose on the contrary to the claim that there are parallel curves $C,C^\prime$ of $\partial F$ in an annulus component $A$ of $\partial E(K) \setminus {\mathcal C}$. All the edges of $\mathcal T$ meeting $C,C^\prime$ are labelled $1$ by construction. Hence none of these edges are in the forest $\mathcal F \subset T$. The trees of the forest meet the annulus $\overline A$ in subtrees $T_1,T_2$ where $T_1$ (respectively $T_2$) contains all the edges except one of the boundary component $C_1$ (respectively $C_2$) of $\overline A$, where $C_1 \cup C_2 = \partial \overline A$. But then we see that both the trees $T_1,T_2$ cannot contain any of the edges which cross $C,C^\prime$ since the edges in the trees are labelled $0$. Hence it is obvious that the forest $\mathcal F$ is not maximal, since it will not reach any vertices of $\partial E(K)$ trapped between $C,C^\prime$ in $A$. This contradiction establishes the claim. 

To complete the bound, if $X$ s non separating, since there is at most one curve of $\partial F$ in any annulus of $\partial E(K) \setminus {\mathcal C}$, the number of curves in $\partial F = {\mathcal C^\prime}$ is at most the number of curves in $X \cap \partial E(K) = {\mathcal C}$. If $X$ is separating, then since $F$ is on one side of $X$, we see that there is at most one curve of $\partial F$ in any annulus of $\partial E(K) \setminus {\mathcal C}$ on the same side of $X$ as $F$, i.e half of all the annuli. So this gives the required bounds and completes the proof. 
\end{proof}

\begin{theorem}\label{homology}
Suppose that $(S^3,K)$ is a non-trivial knot. Then $K$ has a pre-essential surface if and only if $K$ has a properly embedded orientable incompressible and boundary incompressible surface $G$ at an even slope, so that the inclusion $H_1(G, \Bbb{Z}_2) \to H_1(E(K), \Bbb{Z}_2)$ has image zero. 
Moreover if a surface $G$ can be found with two boundary curves, then the pre-essential surface can be chosen with a single boundary curve (i.e a quasi-spanning surface) and vice versa.
\end{theorem}

\begin{proof}
If $K$ has a pre-essential surface $S$, then by definition there is a properly embedded orientable, incompressible and boundary incompressible surface $G$ at the boundary slope of $S$. Note that the surface $G$ is disjoint from a compact non-orientable properly embedded surface $S_0$ by Lemma \ref{geometrically incompressible}. Hence every loop $C$ on $G$ is disjoint from $S_0$ and so the homology class of $C$ is zero in $H_1(E(K),\Bbb{Z}_2)$ since the unique non-zero class of $H_1(E(K),\Bbb{Z}_2)$ is dual to the homology class $[S_0]$ in $H_2(E(K),\partial E(K),\Bbb{Z}_2)$. It is obvious that if $S$ has a single boundary curve then $G$ has two boundary curves. 

Conversely suppose there is a properly embedded orientable incompressible and boundary incompressible surface $G$ at an even boundary slope,  with the property that $H_1(G, \Bbb{Z}_2) \to H_1(E(K),\Bbb{Z}_2)$ has image zero. Using lemma \ref{normal}, we can construct a properly embedded non-orientable surface $S$ satisfying $S \cap G = \emptyset$, since $H_1(G, \Bbb{Z}_2) \to H_1(E(K),\Bbb{Z}_2)$ has image zero. So we conclude that there is a compact properly embedded non-orientable surface $S$ in $E(K)$ which is disjoint from $G$. Note that $S$ has a single boundary curve if $G$ has two boundary curves.

But now it is easy to deduce that $S$ is pre-essential. As in Lemma \ref{pre-essential}, we can perform a series of compressions of $S$ to obtain a new geometrically incompressible and boundary incompressible surface $S_0$ which is disjoint from $G$. It is now easy to deduce that $S_0$ is also algebraically boundary incompressible, since any boundary compression can be homotoped off of $G$. Hence $S_0$ is pre-essential as required. 
\end{proof}

\begin{remark}
Theorem \ref{homology} shows that the even boundary slope conjecture is very closely related to the existence of a pre-essential surface. If a knot has a pre-essential surface then it satisfies the even boundary slope conjecture and the additional assumption about $\Bbb{Z}_2$ homology implies the converse statement. 
\end{remark}

\begin{example}
A nice class of examples is the $(p,q)$-torus knots $K_{p,q}$. These knots all satisfy the Neuwirth conjecture with Neuwirth surface given by an essential annulus. Since the boundary slope of an essential annulus for $K_{p,q}$ is $pq$, it follows that the boundary slope of the annulus is even if one of $p,q$ are even and is odd if both $p,q$ are odd. We can therefore apply Theorem \ref{homology} to deduce that even torus knots have pre-essential surfaces, since it is obvious that the annulus has image zero in $\Bbb{Z}_2$ homology. 

It is also not difficult to show that the odd torus knots do not have pre-essential surfaces. One way of proving this is to observe that if there was such a surface, then there would have to be a properly embedded orientable incompressible and boundary incompressible surface at an even boundary slope, by Lemma \ref{pre-essential}.  But torus knots have no orientable incompressible surfaces other than annuli and the unique minimal genus Seifert surfaces, which have longitudinal boundary slopes. So we conclude that the only separating essential surfaces are annuli and so no pre-essential surface is possible in the odd case. 
\end{example}

\subsection{Large knots}

In this subsection, we show that for large knots, the Neuwirth conjecture is satisfied if a closed incompressible surface can be found in the knot complement with suitable homological properties. 

\begin{theorem}\label{large}
Suppose that $(S^3,K)$ is a knot and there is an embedded closed incompressible surface $\Sigma$ in $E(K)$. Suppose that $i:\Sigma \to E(K)$ is the inclusion map and the induced map $i_*:H_1(\Sigma, \Bbb{Z}) \to H_1(E(K),\Bbb{Z})$ has non-zero image, but the induced map ${\tilde i}_*:H_1(\Sigma, \Bbb{Z}_2) \to H_1(E(K),\Bbb{Z}_2)$ has zero image. Then $K$ satisfies the Neuwirth conjecture. 
\end{theorem}

\begin{proof}
The first step is to construct a properly embedded orientable incompressible and boundary incompressible surface $S$ with the property that $\partial S \subset \partial E(K) \cup \Sigma$ and $\text{int}S\cap \Sigma = \emptyset$. We also require that the boundary slope of $S \cap \partial E(K)$ is a non-zero even integer. 

Consider the $3$-manifold $M$ obtained by splitting $E(K)$ open along $\Sigma$ and discarding the component which has boundary given by $\Sigma$. Hence $\partial M = \partial E(K) \cup \Sigma$. Let $W_1,W_2$ be the images of $j_*:H_1(\partial E(K),\Bbb{Z}) \to H_1(M,\Bbb{Z})$ and of $i_*:H_1(\Sigma,\Bbb{Z}) \to H_1(M,\Bbb{Z})$ respectively, where $j,i$ are the inclusion maps. By the assumption that $i_*:H_1(\Sigma, \Bbb{Z}) \to H_1(E(K),\Bbb{Z})$ has non-zero image, we see that $W_2$ must contain a non-zero multiple of the homology class of the meridian $[C]$ , which generates $H_1(E(K),\Bbb{Z})$. On the other hand, the assumption that ${\tilde i}_*:H_1(\Sigma, \Bbb{Z}_2) \to H_1(E(K),\Bbb{Z}_2)$ has zero image, implies that this multiple must be even. Consider next the homology class of the longitude $[C']$ in $H_1(M,\Bbb{Z})$. If any non-zero multiple $n[C']$ is zero, then we see that there is a 2-chain $c$ with boundary $n[C']$ in $M$. But then $W_2$ cannot contain any non-zero multiples of the homology class of the meridian $[C]$ since such classes have non-zero intersection number with the 2-chain $c$, but $\Sigma$ is disjoint from $c$. So we get a contradiction. We conclude that the map $j_*$ is one-to-one, since the only classes which could possibly be in the kernel are multiples of the longitude. This establishes that $W_1$ is a rank $2$ subgroup with basis consisting of the meridian, longitude pair. 

Next, we can construct a map $f:M \to S^1$ using the cohomology class in $H^1(M,\Bbb{Z})$ of the meridian $C$. Making $f$ transverse to a base point $x_0 \in S^1$, we get that $f^{-1}(x_0)$ is a properly embedded orientable surface $S'$ with $\partial S' \cap \partial E(K) = C'$. We can surger the map $f$ as in the classical method of Stallings (\cite{S}, \cite{AR}). After a finite number of homotopies of $f$, the result is that $S'$ is replaced by an incompressible and boundary incompressible surface, which still satisfies $\partial S' \cap \partial E(K) = C'$. We again denote this surface by $S'$. Notice that since $W_2$ contains an even multiple of the homology class of the meridian $[C]$, we see that $S' \cap \Sigma \ne \emptyset$. For by the same argument as in the previous paragraph, if $S' \cap \Sigma = \emptyset$, then the 2-chain carried by $S'$ would have non-zero interseciton number with $[C]$, which gives a contradiction. But then the sum of the homology classes of the loops of $S' \cap \Sigma$ equals $[C']$ in $H_1(M)$. This shows  that $W_2$ contains a subgroup of rank $2$, with basis consisting of a non-zero even multiple $m[C]$ and $[C']$.

We can now complete the construction of $S$. Since $W_1,W_2$ both contain the homology class $m[C] +[C']$ there is a 2-chain $c^*$ with boundary given by curves on $\partial E(K)$ and $\Sigma$ which represent $m[C] +[C']$. We can then build a new cohomology class in $H^1(M)$ which is dual to this chain. Construct a new map $f^*:M \to S^1$ using this cohomology class, arranging that ${f^*}^{-1}(x_0)$ is a properly embedded orientable surface $S$ with $\partial S \cap \partial E(K) = C^*$, where $C^*$ has slope $m[C] +[C']$. This surface also has some boundary curves on $\Sigma$. By surgering the map $f^*$, we can homotop $f^*$ until the inverse image ${f^*}^{-1}(x_0)$ is incompressible and boundary incompressible. So this completes the construction of $S$.

To finish the proof, we use Lemma \ref{normal}. Namely we build a pre-essential surface $S^\prime$ which is disjoint from $S$.  Consider the homomorphism $\alpha:H_1(M,\Bbb{Z}_2) \to H_1(E(K),\Bbb{Z}_2) =\Bbb{Z}_2$ induced by the inclusion map of $M \subset E(K)$.  
Clearly $\text{kernel} (\alpha)$ contains all the homology classes of loops in $S$ and in $\Sigma$. Therefore we can choose $X=S \cup \Sigma$ and apply Lemma \ref{normal} giving a compact properly embedded non-orientable surface $S^\prime$ with a connected essential boundary disjoint from $X$.  It is easy to see that $S^\prime$ is pre-essential. For any possibly singular boundary compression of $S^\prime$ can be pulled off the incompressible surface $\Sigma$ and then off the incompressible and boundary incompressible surface $S$. 

The final step is to compress $\partial N(S^\prime)$. Note that this can be done in the complement of $S \cup \Sigma$ but we don't actually need this. We end up with a properly embedded incompressible and boundary incompressible surface $F$ with two boundary curves at non-zero even integer boundary slope. So this implies the knot satisfies the Neuwirth conjecture. 
\end{proof}

\begin{remark}
In \cite{O5}, it is shown that if $(S^3,K)$ is a large knot and it contains a closed embedded incompressible surface $\Sigma$ so that the image
$H_1(\Sigma,\Bbb{Z}) \to H_1(E(K),\Bbb{Z})$ is zero, then there is a Seifert surface $F$ for $K$ disjoint from $\Sigma$. Conversely if a knot $(S^3,K)$  has a Seifert surface $F$ which is not free, then there is a closed embedded incompressible surface $\Sigma$ disjoint from $F$ with the property that $H_1(\Sigma,\Bbb{Z}) \to H_1(E(K),\Bbb{Z})$ is zero.
The method of Theorem \ref{large} gives a similar result for non-orientable spanning surfaces.
\end{remark}

\begin{example}
We construct many examples of knots $(S^3,K)$ with embedded incompressible surfaces $\Sigma$, satisfying the image of $H_1(\Sigma,\Bbb{Z}_2) \to H_1(E(K),\Bbb{Z}_2)$ is zero but the image of $H_1(\Sigma,\Bbb{Z}) \to H_1(E(K),\Bbb{Z})$ is non-zero. Construct a knotted handlebody $Y$ in $S^3$, i.e an embedding so that $\overline{S^3 \setminus Y}=M$ is a $3$-manifold with incompressible boundary $\Sigma=\partial Y$. Choose an embedded knot $K \subset Y$ with the following three properties. Firstly, $K$ is disk busting in $Y$, i.e meets every compressing disk in $Y$. Secondly the induced map by inclusion $H_1(K,\Bbb{Z}_2) \to H_1(Y,\Bbb{Z}_2)$ has image zero, but the induced map by inclusion $H_1(K,\Bbb{Z}) \to H_1(Y,\Bbb{Z})$ has image non-zero. Finally the inclusion induces a map $H_1(K,\Bbb{Z}) \to H_1(Y,\Bbb{Z})$ with non-zero image. Then it is easy to see that $\Sigma$ is a closed incompressible surface in $E(K)$ satisfying the required properties as in Theorem \ref{large}. Note that there are many such knots which are parallel into $\Sigma$, i.e are embedded in $Y$ so that the knot is isotopic into $\Sigma = \partial Y$.
\end{example}

\subsection{Degree one maps}

\begin{definition}
A degree one map $\phi$ between knots $(S^3,K)\to (S^3,K')$  is a continuous proper map between knot complements $E(K) \to E(K')$ which induces an isomorphism $H_3(E(K),\partial E(K)) \to H_3(E(K'),\partial E(K'))$. 
\end{definition}

\begin{theorem}\label{degreeone}
If there exists a degree one map $(S^3,K)\to (S^3,K')$ and $K'$ satisfies the Neuwirth Conjecture, then $K$ also satisfies the Neuwirth Conjecture.
\end{theorem}

\begin{proof}
We follow the method of surgering maps as in Theorem \ref{large}. Choose a Neuwirth surface $F'$ for the knot $K'$ and assume that $\phi$ is a degree one map $(S^3,K)\to (S^3,K')$. We first homotop $\phi$ so that it is transverse to $F'$. We can also assume after a homotopy that $\phi$ restricts to  a homeomorphism $\partial E(K) \to \partial E(K')$ since it is easy to check that the induced map on the boundary tori is degree one. Then the inverse image $F_0$ of $F'$ under $\phi$ is a properly embedded compact orientable surface with two boundary curves, which have the same boundary slope as $F$, since as is well-known, a degree one map maps the longitude, meridian pair for $K$ to the longitude, meridian pair for $K'$. Although $F_0$ need not be connected, it must have a connected component $F_1$ containing the two boundary curves. For suppose there are two components $F_2,F_3$ each with one boundary curve. The map $\phi$ restricted to $F_2$ gives an induced map $H_1(F_2) \to H_1(F')$ which maps the class $[\partial F_2] =0$ in $H_1(F_2)$ to a non zero class in $H_1(F')$, which is of the form $[C]$ where $C$ is one of the components of $\partial F'$. This is a contradiction. 

If $F_1$ is incompressible and boundary incompressible, it is clearly a Neuwirth surface for $K$. Since $F_1$ has two boundary components, if it is boundary compressible then it is compressible. Choose a compressing disk $D$ for $F$. We can now perform Stallings technique of surgering the map $\phi$. Namely a homotopy of $\phi$ can be performed so that the inverse image of $F'$ is the result of compressing $F$ along $D$.  After a finite number of surgeries, we must end up with an incompressible component $F^*$ of $\phi^{-1}(F')$  which has the same boundary as $F_1$. But then $F^*$ is incompressible and boundary incompressible and hence is the required Neuwirth surface for $K$.
\end{proof}

\begin{corollary}\label{pre-essentialpullback}
If there exists a degree one map $(S^3,K)\to (S^3,K')$ and $K'$ has a pre-essential surface, then $K$ also has a pre-essential surface.
\end{corollary}

\begin{proof}
Suppose that $S'$ is a pre-essential surface for $K'$.  By definition, if $G'$ is the boundary of $N(S')$ in $E(K')$, then compressing $G'$ produces a properly embedded orientable incompressible and boundary incompressible surface $F'$ for $K'$. Applying Theorem \ref{degreeone}, we see that the degree one map $\phi$ from $(S^3,K)\to (S^3,K')$ can be homotoped until the inverse image of $F'$ contains a component which is a properly embedded orientable incompressible and boundary incompressible surface $F$ for $K$. 

Now it is easy to also arrange that the inverse image of $S'$ contains a non-orientable quasi-spanning surface $S$ for $K$. In fact, recall from the proof of Theorem \ref{degreeone}, $\phi$ can be assumed to restrict to a homeomorphism between the peripheral tori preserving the longitude, meridian pair. Therefore the inverse image of $S'$ will contain exactly one boundary curve at the same slope as that of $S'$. Hence if the boundary slope of $S'$ is not the longitude, then $S$ must be a non-orientable surface. But then $S$ is disjoint from $F$ and so if we compress the boundary of $N(S)$ in $E(K)$, all the compressions can be isotoped off $F$. We see that $S$ is pre-essential, since compressing $\partial N(S)$ in $E(K)$ will produce a properly embedded orientable incompressible and boundary incompressible surface (not necessarily the same as $F$). 

Finally in the case that the pre-essential surface $S'$ has longitudinal boundary slope, then the inverse image must still contain a non-orientable spanning surface $S$. For the map induced by $\phi$ from $S'$ to $S$ must have degree one , so that the induced map $H_2(S', \partial S') \to H_2(S, \partial S)$ is an isomorphism. But this is impossible if $S$ is non-orientable and $S'$ is orientable. So this completes the proof. 
\end{proof}

\begin{corollary}\label{simple}
Suppose a knot $K$ satisfies the Neuwirth conjecture or has a pre-essential surface, respectively. Then any satellite knot $K^*$ obtained from $K$ and any sum $K \# K^\prime$ satisfy the Neuwirth conjecture or has a pre-essential surface, respectively. 
\end{corollary}

\begin{proof}
By Theorem \ref{degreeone} since there is clearly a degree one map from a satellite knot $K^*$ to the knot $K$, if $K$ satisfies the Neuwirth conjecture then so does $K^*$. Similarly by Corollary \ref{pre-essentialpullback}, if $K$ has a pre-essential surface then so does the satellite knot $K^*$ using the degree one map $K^* \to K$. 

The argument for a sum $K \# K^\prime$ is similar. We just observe that there is a degree one map $K \# K^\prime \to K$. 
\end{proof}

It is a fundamental problem whether a topological invariant becomes ``smaller'' under a degree one map between knots (\cite[Question 3.1]{W}).
It is known that the rank of the fundamental group, Gromov's simplicial volume, the Haken number of incompressible surfaces, the knot genus, the homology group and the Alexander polynomial behave in this way under degree one maps.
The next Theorem \ref{boundary slopes} is a new addition to this list of invariants and it can be used to determine whether there does not exist a degree one map between two knots, using the calculation of boundary slopes of knots in \cite{HT2}, \cite{HO}, \cite{D} and recently \cite{DG}.

\begin{theorem}\label{boundary slopes}
If there exists a degree one map $(S^3,K)\to (S^3,K')$, then the set of boundary slopes of $K$ includes the set of boundary slopes of $K'$.
\end{theorem}

\begin{proof}
By the methods of Theorem \ref{degreeone} and Corollary \ref{pre-essentialpullback}, if there is a properly embedded orientable incompressible and boundary incompressible surface with a given boundary slope in $E(K')$, then by pulling this back under the degree one map and performing surgery of the map, we obtain a similar surface in $E(K)$ with the same boundary slope. Hence the set of boundary slopes of $K'$ is a subset of the set of boundary slopes of $K$.
\end{proof}

\begin{remark}
In \cite[Corollary 2.5]{HS}, a similar result was obtained for 2-bridge knots and epimorphisms.
\end{remark}

\begin{corollary}
If there exists a degree one map $(S^3,K)\to (S^3,K')$ and $K'$ satisfies the (strong) even boundary slope conjecture, then $K$ also satisfies the (strong) even boundary slope conjecture.
\end{corollary}

\begin{corollary}\label{simple2}
If $K$ satisfies the (strong) even boundary slope conjecture, then any satellite knot or sum of a knot with $K$ also satisfies the (strong) even boundary slope conjecture.
\end{corollary}

\begin{proof}
The proof follows along the same lines as Corollary \ref{simple}.
\end{proof}

\begin{example}
Take a $2$-component boundary link $L \subset S^3$ with the property that an arc $\lambda$ between the two components of $L$ can be found which misses a choice of two disjoint spanning surfaces for the components of $L$. We also assume that the handlebody $Y = N(L \cup \lambda)$ is knotted, i.e $\partial Y$ is incompressible in $\overline{S^3 \setminus Y}$. Next pick any knot $K$ embedded in $Y$ which is disk busting in $Y$. We claim that the knot $(S^3, K)$ has a degree one map to the knot $(S^3,K')$ where the pair $(Y',K')$ is homeomorphic to the pair $(Y,K)$ and $Y' \subset S^3$ is unknotted. 

The reason is that the two disjoint spanning surfaces in $\overline{S^3 \setminus Y}$ can be mapped to disks and then the closure of the complement of these spanning surfaces mapped to a ball, so that there is a degree one map $\overline{S^3 \setminus Y} \to Y_1$ which is a genus $2$ handlebody with meridian disks with the same boundaries as the spanning surfaces. This map clearly extends to a degree one map $E(K) \to E(K')$ as claimed. 
\end{example}

\subsection{Pretzel surfaces and checkerboard surfaces}

A pretzel knot or link $P(p_1,\ldots,p_n)$ bounds a checkerboard surface which consists of two disks and $p_i$ half-twisted bands, and we call it a {\em pretzel surface}.
The following theorem determines which pretzel surfaces are algebraically incompressible and boundary incompressible.

\begin{theorem}\label{pretzel}
Let $K$ be a pretzel knot or link $P(-p_1,p_2,\ldots,p_n)$ with $p_i\ge 2$ and $n\ge 3$.
Then the pretzel surface for $K$ is algebraically incompressible and boundary incompressible if and only if $(-p_1,p_2,\ldots,p_n)\ne (-2,3,3)$, $(-2,3,4)$, $(-2,3,5)$ nor $(-2,2,p_3)$, where $p_3$ is odd.
\end{theorem}

\begin{proof}
Suppose that $K$ does not bound an algebraically incompressible and boundary incompressible pretzel surface $F$.
By taking a regular neighbourhood of $F$, we obtain a genus $n-1$ Heegaard splitting $V_1\cup_S V_2$, where $V_1=N(F)$ and $S=\partial N(F)$.
We assume that $V_1$ intersects the 2-sphere $S^2$, which includes a $(-p_1,p_2,\ldots,p_n)$-pretzel diagram of $K$, in a $n$-punctured sphere.
Thus $V_2$ intersects $S^2$ in $n$ disks $D_1,D_2,\ldots,D_n$, where $D_1$ is sandwiched between the $-p_1$ half-twisted band and the $p_2$ half-twisted band, $D_2$ is sandwiched between the $p_2$ half-twisted band and the $p_3$ half-twisted band, and $D_n$ is an outside disk between the $p_n$ half-twisted band and the $-p_1$ half-twisted band.
The knot $K$ can be slightly isotoped onto the genus $n-1$ Heegaard surface $S$ so that $|K\cap \partial D_1|=(p_1-1)+(p_2-1)$, $|K\cap \partial D_2|=p_2+p_3$ and $|K\cap \partial D_n|=(p_n-1)+(p_1-1)$.
See Figure \ref{(-2,3,5)-pretzel} for $K=P(-2,3,5)$.

\begin{figure}[htbp]
	\begin{center}
	\includegraphics[trim=0mm 0mm 0mm 0mm, width=.6\linewidth]{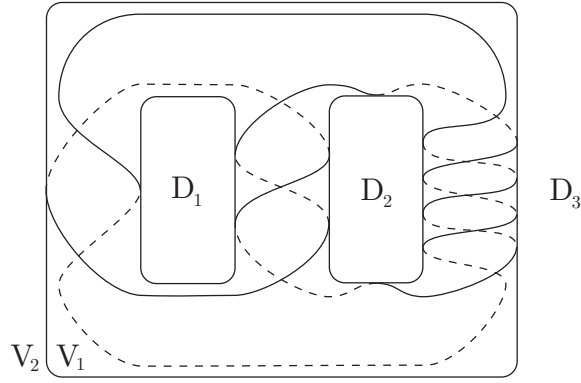}
	\end{center}
	\caption{$(-2,3,5)$-pretzel knot on a genus two Heegaard surface $S$}
	\label{(-2,3,5)-pretzel}
\end{figure}

If we cut $V_2$ along $D_1\cup D_2\cup\cdots\cup D_n$, then we have two 3-balls $B_1$ and $B_2$.
On $\partial B_i$, we have a graph $G_i$ with fat vertices as copies of $D_1,D_2,\ldots,D_n$ and edges from $K\cap \partial B_i$.
Figure \ref{graph} illustrates the graph $G_i$ for $K=P(-p_1,p_2,\ldots,p_n)$, where the number indicated at each edge denotes the number of multiple edges.

\begin{figure}[htbp]
	\begin{center}
	\includegraphics[trim=0mm 0mm 0mm 0mm, width=.6\linewidth]{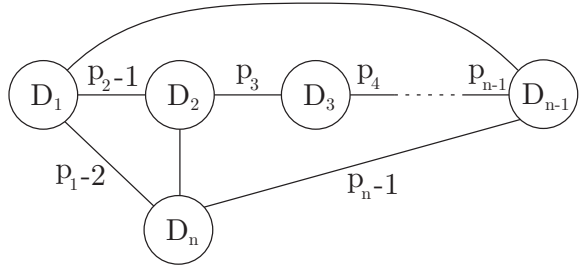}
	\end{center}
	\caption{The graph $G_i$ for $K=P(-p_1,p_2,\ldots,p_n)$}
	\label{graph}
\end{figure}

\begin{claim}\label{p_1}
$p_1=-2$ and $n=3$.
\end{claim}

\begin{proof}
Let $D$ be a compressing disk for $S-K$.
We assume that $|D\cap (D_1\cup D_2\cup\cdots\cup D_n)|$ is minimal up to isotopy of $D$.
If $|D\cap (D_1\cup D_2\cup\cdots\cup D_n)|=0$, then $D$ is parallel to some $D_i$ in $V_2$.
However, this is not possible since each $\partial D_i$ intersects $K$ at least two points.
Hence there exists an outermost disk $\delta$ in $D$ with respect to $D\cap (D_1\cup D_2\cup\cdots\cup D_n)$.
Then the boundary of $\delta$ shows that a graph $G_i$ has a cut vertex.
Since $p_i\ge 2$, it follows that $p_1-2=0$ and $n=3$.
In this case, $D_2$ is a cut vertex.
\end{proof}

Without loss of generality, we may assume that $p_2\le p_3$.

\begin{claim}\label{p_2}
$p_2\le 3$.
\end{claim}

\begin{proof}
We remark that there is a possibility for outermost disks in $D$ with respect to $D\cap (D_1\cup D_2\cup D_3)$ to be in both sides of $B_1$ and $B_2$.
However, if $p_2\ge 4$, then any outermost disk $\delta$ in $D$ with respect to $D\cap (D_1\cup D_2\cup D_3)$ is contained in only one side of $B_1$ and $B_2$ and hence is unique.
In this case, the next arc to any outermost arc is contained in $D_1$ or $D_3$.
Therefore, any two outermost arcs cannot be next to each other.
Moreover, the second outermost arc parallel to an outermost arc cannot exist since if one end of it is contained in $D_1$, then it follows that the another end is contained in $D_3$.
\end{proof}

\begin{claim}\label{p_2=3}
If $p_2=3$, then $p_3\le 5$.
\end{claim}

\begin{proof}
Since $p_1=-2$ and $p_2=3$, it follows that $|K\cap \partial D_1|=3$.
Therefore, by exchanging $D$ if necessary, there exists a compressing or boundary compressing disk $D$ which is disjoint from $D_1$.
We cut $V_2$ along $D_1$, and obtain a solid torus, say $V_2'$, and the rest is a solid torus $V_1'$.
On the Heegaard torus $S'=\partial V_i'$, we have a theta-curve graph with two fat vertices as two copies $D_1', D_1''$ of $D_1$ and three edges $k_1,k_2,k_3$ as in Figure \ref{(-2,3,5)-theta}.

\begin{figure}[htbp]
	\begin{center}
	\includegraphics[trim=0mm 0mm 0mm 0mm, width=.6\linewidth]{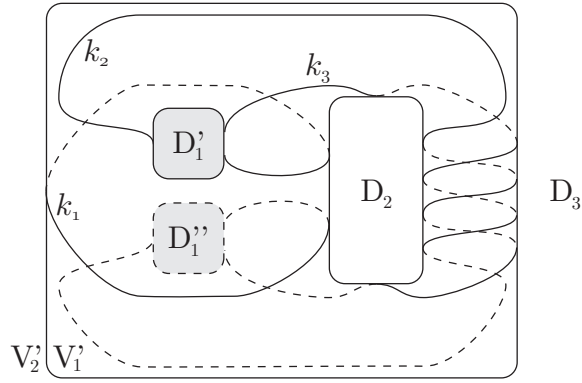}
	\end{center}
	\caption{A theta curve graph on a Heegaard torus $S'$}
	\label{(-2,3,5)-theta}
\end{figure}

If $p_3\ge 6$, then the constituent knot $k_1\cup k_3$ goes around $V_2'$ at least twice.
However this shows that $\partial D$ intersects $k_1\cup k_3$ in at least two points, and contradicts that $D$ is a compressing or boundary compressing disk.

Otherwise, if $p_3\le 5$, then there exists a compressing or boundary compressing disk for $S'-(k_1\cup k_2\cup k_3)$ in $V_2'$.
\end{proof}

\begin{claim}\label{p_2=2}
If $p_2=2$, then $p_3$ is odd.
\end{claim}

\begin{proof}
Since $p_1=-2$ and $p_2=2$, it follows that $|K\cap \partial D_1|=2$.
Therefore, by exchanging $D$ if necessary, there exists a compressing or boundary compressing disk $D$ which is disjoint from $D_1$.
As in the previous claim, we cut $V_2$ along $D_1$, and obtain a solid torus, say $V_2'$, and the rest is a solid torus $V_1'$.
On the Heegaard torus $S'=\partial V_i'$, we have a loop corresponding to $k_1$ and a $2$-cycle with two fat vertices as two copies $D_1', D_1''$ of $D_1$ and two edges $k_2,k_3$ similarly as in Figure \ref{(-2,3,5)-theta}.
We note that the loop $k_1$ goes around $V_2'$ exactly once.

If $p_3$ is even, then the $2$-cycle $k_2\cup k_3$ is parallel to the loop $k_1$ in the Heegaard torus $S'$, hence it also goes around $V_2'$ exactly once.
Therefore, in this case, $\partial D$ intersects $k_1\cup k_2\cup k_3$ in at least two points, and contradicts that $D$ is a compressing or boundary compressing disk.

Otherwise, in the case that $p_3$ is odd, the $2$-cycle $k_2\cup k_3$ bounds a disk in the Heegaard torus $S'$ and this shows that $S'-(k_1\cup k_2\cup k_3)$ is compressible in $V_2'$.
\end{proof}

The above claims show that if the pretzel surface for $K$ is not algebraically incompressible and boundary incompressible, then $(-p_1,p_2,\ldots,p_n)= (-2,3,3)$, $(-2,3,4)$, $(-2,3,5)$ or $(-2,2,p_3)$, where $p_3$ is odd.
Conversely, if $(-p_1,p_2,\ldots,p_n)= (-2,3,3)$, $(-2,3,4)$, $(-2,3,5)$ or $(-2,2,p_3)$, where $p_3$ is odd, then there exists a compressing or boundary compressing disk for $S-K$ in $V_2$ as shown in Claims \ref{p_2=3} and \ref{p_2=2}.
This completes the proof of Theorem \ref{pretzel}.
\end{proof}

Theorem \ref{pretzel} on pretzel surfaces can be generalized to checkerboard surfaces as follows.
Let $G$ be a planar graph in the 2-sphere $S^2$ whose edges $e_1,\ldots,e_n$ have weights $w_1,\ldots,w_n \in \Bbb{Z}$.
We replace each vertex $v_i$ of $G$ with a disk $d_i$, and replace each edge $e_i$ with a $w_i$ half-twisted band $b_i$.
Then we obtain a surface $F_G$ and a knot or link $K_G=\partial F_G$.
Naturally, $K_G$ has a diagram on $S^2$ and $F_G$ can be regarded as a checkerboard surface for $K_G$.
We remark that any knot or link can be obtained from a suitable weighted planar graph in this manner.
For example, a pretzel knot or link $P(p_1,\ldots,p_n)$ can be obtained as $K_G$ from a theta-curve graph $G$ whose multiple edges $e_1,\ldots,e_n$ have weights $p_1,\ldots,p_n$, and the pretzel surface coincides with $F_G$.

\begin{theorem}\label{checkerboard}
Let $G$ be a 2-connected planar graph in $S^2$ with edges $e_1,\ldots,e_n$ having weights $w_1,\ldots,w_n \in \Bbb{Z}$.
\begin{enumerate}
\item If $|w_i|\ge 3$ for all $i$, then the surface $F_G$ is algebraically incompressible and boundary incompressible.
\item If $w_1\le -2$ and $w_i\ge 2$ for $i=2,\ldots,n$, then $G$ has multiple two parallel edges $e_1$, say $e_2$, such that $w_1=-2$ and $w_2=2$ or $3$.
\end{enumerate}
\end{theorem}

\begin{proof}
We can prove Theorem \ref{checkerboard} by the argument similar to Theorem \ref{pretzel}, so we omit the details.
Similarly to the proof of Theorem \ref{pretzel}, we obtain a graph $G_i$ on the 3-ball $B_i$ and if the condition in (1) is satisfied, then $G_i$ is 2-connected.
This shows that there exists no compressing disk for $\partial N(F_G)-K_G$.

As in Claim \ref{p_1}, it follows from the existence of an outermost disk $\delta$ that $w_1=-2$ and there exists a region of $S^2-N(F_G)$ as $D_2$ containing the outermost arc.
Namely, $D_2$ can be enlarged into a region $D_2'$ of $S^2-G$ which contacts the two end points of $e_1$ but does not contacts the interior of $e_1$.

As in Claim \ref{p_2}, it follows from the existence of outermost disks in both sides of $B_1$ and $B_2$ that an edge adjacent to $e_1$, say $e_2$, must be parallel to $e_1$ and the weight $w_2$ is less than or equal to $3$.
Thus we obtain two multiple edges $e_1$ and $e_2$ satisfying the condition in (2).
\end{proof}

\begin{remark}
Theorem \ref{checkerboard} (2) shows that the diagram of $K_G$ contains a Montesinos tangle $MT(-1/2,1/2)$ or $MT(-1/2,1/3)$.
\end{remark}

\subsection{Montesinos knots}

\begin{theorem}\label{Montesinos}
Montesinos knots satisfy the strong Neuwirth conjecture.
\end{theorem}

\begin{proof}
Let $K=M(r_1,\ldots,r_n)$ be a Montesinos knot with $n$ rational tangles of slope $r_i$.
We may assume that $r_i\ne 0/1, 1/0$ and $n\ge 3$ since a composite knot with $2$-bridge knot summands satisfies the strong Neuwirth conjecture.
Moreover, we may assume that $R_-\le R_+$, where $R_-$ denotes the number of negative $r_i$'s and $R_+$ denotes the number of positive $r_i$'s.
Suppose that $K$ is not a torus knot.
Hence, by \cite{BS}, $K\ne M(-1/2,1/3,1/3)$ nor $M(-1/2,1/3,1/5)$.
We put each rational tangle in a standard form.

If $R_-=0$, then $K$ is alternating and at least one of the two checkerboard surfaces is non-orientable, algebraically incompressible and boundary incompressible.

Hereafter, we assume that $R_-\ge 1$.
Under such a condition, we can take each slope $r_i$ so that $|r_i|<1$.

If $R_-\ge 2$, then $K$ is $+$-adequate and $-$-adequate.
Then, by Theorem \ref{state}, both of the state surfaces $F_{\sigma_+}$ and $F_{\sigma_-}$ are algebraically incompressible and boundary incompressible, and at least one of them is non-orientable.

Henceforth, we assume that $R_-=1$.
By exchanging $r_i$, we can put $K=M(-r_1,r_2,\ldots,r_n)$, where $r_i>0$ for $i=1,2,\ldots,n$.
Then $K$ is $+$-adequate, and the state surface $F_{\sigma_+}$ is algebraically incompressible and boundary incompressible.
Hence if $K$ is not positive, then $F_{\sigma_+}$ is non-orientable.

From now on, we assume that $K$ is positive.
Here, we remark that $K$ is positive if and only if $F_{\sigma_+}$ is orientable.
Therefore, both of the two checkerboard surfaces are non-orientable since they can be obtained as state surfaces except for the positive state $\sigma_+$.

\begin{figure}[htbp]
	\begin{center}
	\includegraphics[trim=0mm 0mm 0mm 0mm, width=.6\linewidth]{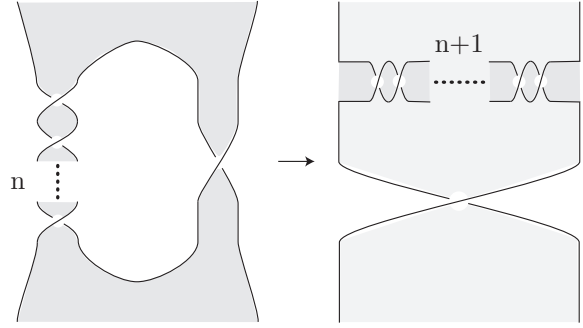}
	\end{center}
	\caption{Deplumbing a checkerboard surface}
	\label{deplumbing}
\end{figure}

By using a deformation as in Figure \ref{deplumbing} inductively, one of the two checkerboard surfaces for $K=M(-r_1,r_2,\ldots,r_n)$ can be deplumbed to a pretzel surface for a pretzel knot or link $K'=P(-\lceil 1/r_1 \rceil, \lceil 1/r_2 \rceil,\ldots, \lceil 1/r_n \rceil)$, where $\lceil x \rceil$ denotes the ceiling function of $x$ which is the smallest integer not less than $x$.

If $K'$ is neither $P(-2,3,3)$, $P(-2,3,4)$, $P(-2,3,5)$ nor $P(-2,2,p)$, where $p$ is odd, then the checkerboard surface for $K$ is non-orientable, algebraically incompressible and boundary incompressible.

Therefore, hereafter we may assume that $K=M(-r_1,r_2,r_3)$ and we need to consider the following two cases.

\begin{enumerate}
\item $K'=P(-2,3,3)$, $P(-2,3,4)$ or $P(-2,3,5)$.
\begin{enumerate}
\item $r_1\ne 1/2$.
\item $r_1=1/2$.
\end{enumerate}
\item $K'=P(-2,2,p)$, where $p$ is odd.
\begin{enumerate}
\item $r_2\ne 1/2$.
\item $r_2=1/2$.
\end{enumerate}
\end{enumerate}

In Case (1)-(a), let $\sigma$ be a state which has all negative signs except for one crossing in the rational tangle with slope $-r_1$.
Then, we have a non-orientable state surface $F_{\sigma}$ for $K$ as in Figure \ref{1-a} which can be obtained from a pretzel surface for $P(-1,3,3)$, $P(-1,3,4)$ or $P(-1,3,5)$.
Since the pretzel surface is a genus one Seifert surface for $3_1$, $4_1$ or $5_2$ up to mirror image respectively, $F_{\sigma}$ is algebraically incompressible and boundary incompressible.

\begin{figure}[htbp]
	\begin{center}
	\includegraphics[trim=0mm 0mm 0mm 0mm, width=.6\linewidth]{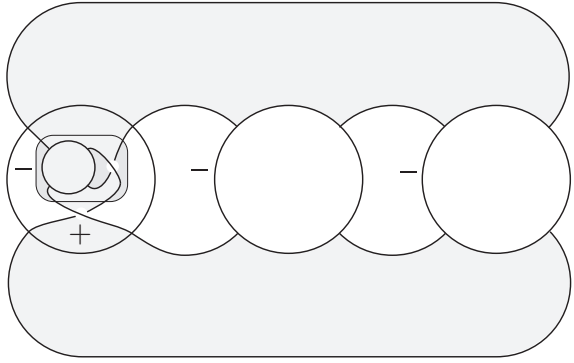}
	\end{center}
	\caption{A state surface $F_{\sigma}$ for $K=M(-r_1,r_2,r_3)$}
	\label{1-a}
\end{figure}

In Case (1)-(b), first we assume that $r_3\ne 1/3$, $1/4$ nor $1/5$.
By using Lemma \ref{deform}, we can deform the Montesinos knot $K=M(-1/2,r_2,r_3)$ into $K=M(1/2,r_2-1,r_3)$.

In the case that $K'=P(-2,3,4)$, then the deformed knot $K=M(1/2,r_2-1,r_3)$ has a minor $P(2,-2,4)$.
By Theorem \ref{pretzel}, the pretzel surface for $P(2,-2,4)$ is algebraically incompressible and boundary incompressible and hence the checkerboard surface for $K=M(1/2,r_2-1,r_3)$ is also algebraically incompressible and boundary incompressible.

Otherwise, for the case that $K'=P(-2,3,3)$ or $P(-2,3,5)$, let $\sigma$ be a state as in Figure \ref{1-b}.
Then, we have a non-orientable state surface $F_{\sigma}$ for $K$ which can be obtained by plumbings from a pretzel surface for $P(2,-2,2)$ or $P(2,-2,4)$.

\begin{figure}[htbp]
	\begin{center}
	\includegraphics[trim=0mm 0mm 0mm 0mm, width=.6\linewidth]{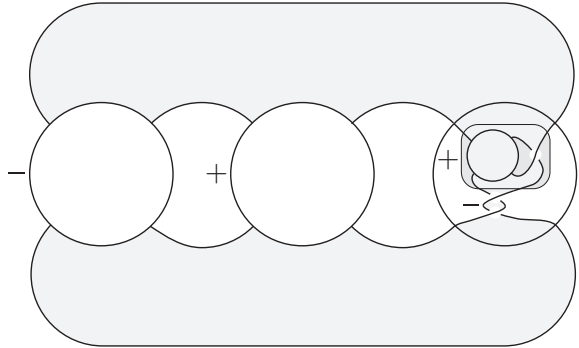}
	\end{center}
	\caption{A state surface $F_{\sigma}$ for $K=M(1/2,r_2-1,r_3)$}
	\label{1-b}
\end{figure}

In Case (2)-(a), by using Lemma \ref{deform}, we can deform the Montesinos knot $K=M(-r_1,r_2,r_3)$ into $K=M(1-r_1,r_2-1,r_3)$.
Then, one of the two checkerboard surfaces for $K=M(1-r_1,r_2-1,r_3)$ can be deplumbed into a pretzel surface for $P(p_1,-p_2,p_3)$, where $p_i\ge 2$.
Since $r_2\ne 1/2$ and the state surface $F_{\sigma_+}$ for $K=M(-r_1,r_2,r_3)$ is orientable, for the corresponding slope $t_2=-r_2/(r_2-1)$ of a subtangle $T_2$ in Lemma \ref{deform}, $\lfloor t_2\rfloor$ is an even integer greater than or equal to $2$.
Hence, $p_2\ge 3$ and the checkerboard surface for $K=M(1-r_1,r_2-1,r_3)$ is algebraically incompressible and boundary incompressible.

In Case (2)-(b), by using Lemma \ref{deform}, we can deform the Montesinos knot $K=M(-r_1,1/2,r_3)$ into $K=M(1-r_1,-1/2,r_3)$.
Since $K$ is connected, $r_1\ne -1/2$.
Hence, we have $r_1>1/2$, $1-r_1<1/2$ and $\lceil 1/(1-r_1)\rceil\ge 3$.
Therefore, the checkerboard surface for $K$ cannot be deplumbed to one for $K'=P(q,-2,p)$, where $q\ge 3$ and $p$ is odd.
Eventually, we arrive at Case (1).
\end{proof}

\begin{example}\label{10_128}
$10_{128}$ bounds a pre-essential checkerboard surface.
Figure \ref{10_128a} shows a regular neighbourhood of a checkerboard surface for $10_{128}=M(3/7,-1/2,1/3)$.

\begin{figure}[htbp]
	\begin{center}
	\includegraphics[trim=0mm 0mm 0mm 0mm, width=.6\linewidth]{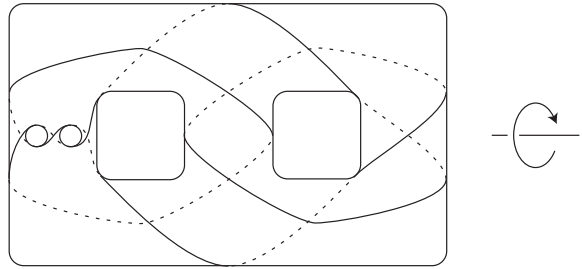}
	\end{center}
	\caption{A regular neighbourhood of a checkerboard surface for $10_{128}=M(3/7,-1/2,1/3)$}
	\label{10_128a}
\end{figure}

After twisting the right-hand half, a compressing disk appears in the outside region as in Figure \ref{10_128b}.

\begin{figure}[htbp]
	\begin{center}
	\includegraphics[trim=0mm 0mm 0mm 0mm, width=.5\linewidth]{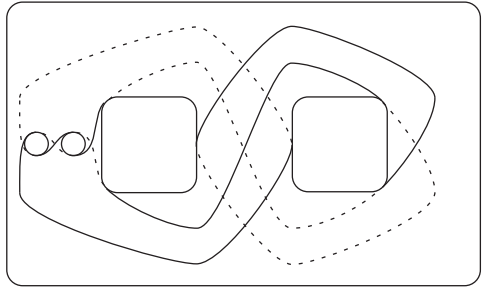}
	\end{center}
	\caption{$10_{128}$ on a genus $4$ closed surface of boundary slope $16$}
	\label{10_128b}
\end{figure}

Therefore, a checkerboard surface for a Montesinos knot $M(3/7,-1/2,1/3)$ is pre-essential.
Similarly, we can observe that a $(-2,3,3)$-pretzel knot, equivalently the $(3,4)$-torus knot, bounds a pre-essential checkerboard surface.
\end{example}

\begin{example}\label{10_139}
$10_{139}$ bounds a $2$-pre-essential checkerboard surface.
Let $F$ be a checkerboard surface for $10_{139}=M(1/3,-3/4,1/3)$ and $V_1$ be a regular neighbourhood of $F$ as in Figure \ref{10_139a}.
Put $S=\partial V_1$ and $V_2=S^3-\text{int}V_1$.
Then, similarly to the case of $10_{128}$, there exists a compressing disk $A$ for $S-K$ in $V_2$ whose boundary is denoted by $a$.

\begin{figure}[htbp]
	\begin{center}
	\includegraphics[trim=0mm 0mm 0mm 0mm, width=.5\linewidth]{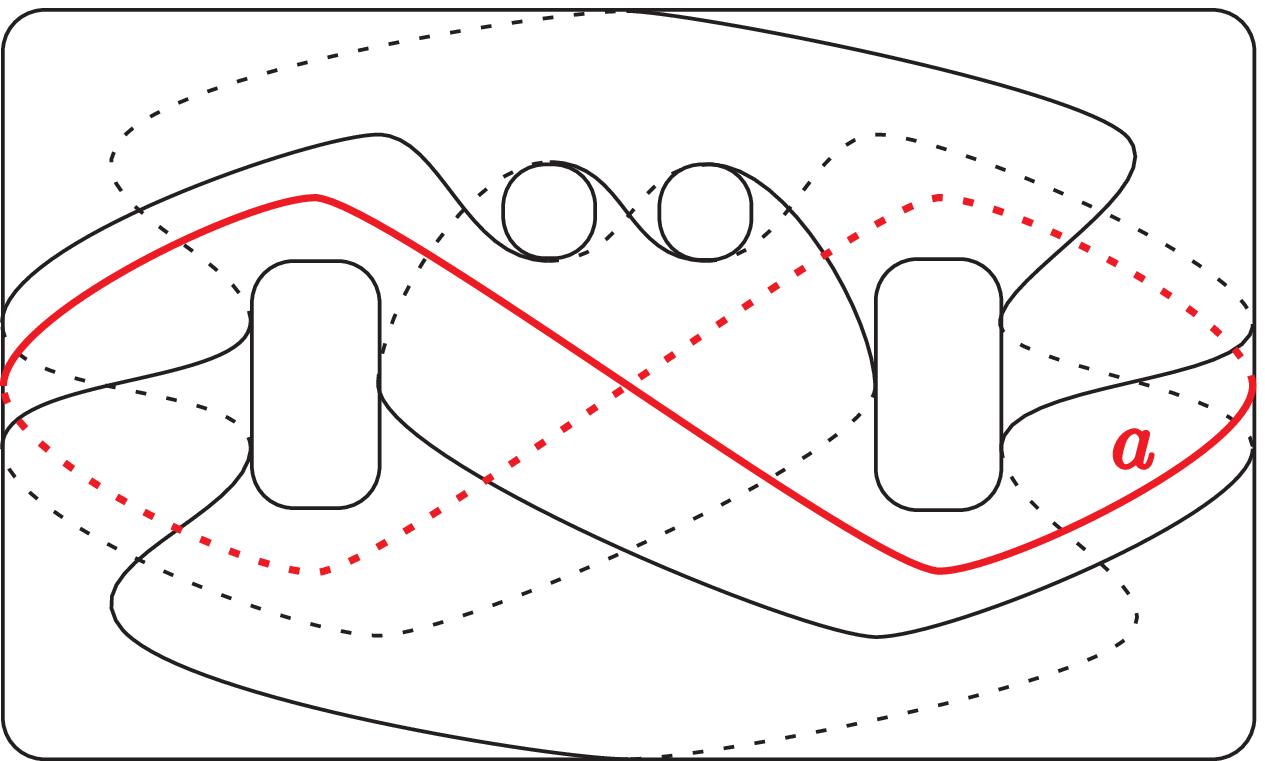}
	\end{center}
	\caption{$10_{139}$ on a genus $4$ closed surface}
	\label{10_139a}
\end{figure}

After compressing $S$ along $A$, we obtain a genus $3$ closed surface $S'$ which separates $S^3$ into $V_1'$ and $V_2'$ respectively.
Then, there exists a second compressing disk $B$ for $S'-K$ in $V_1'$ whose boundary is denoted by $b$.
We remark that this compressing disk $B$ can not be found in the original pair $(S,K)$, however it appears after isotoping a portion of $K$ as in Figure \ref{10_139b}.
This isotopy can be done by sliding a portion of $K$ along the compressing disk $A$ in $S\cup A$.

\begin{figure}[htbp]
	\begin{center}
	\includegraphics[trim=0mm 0mm 0mm 0mm, width=.5\linewidth]{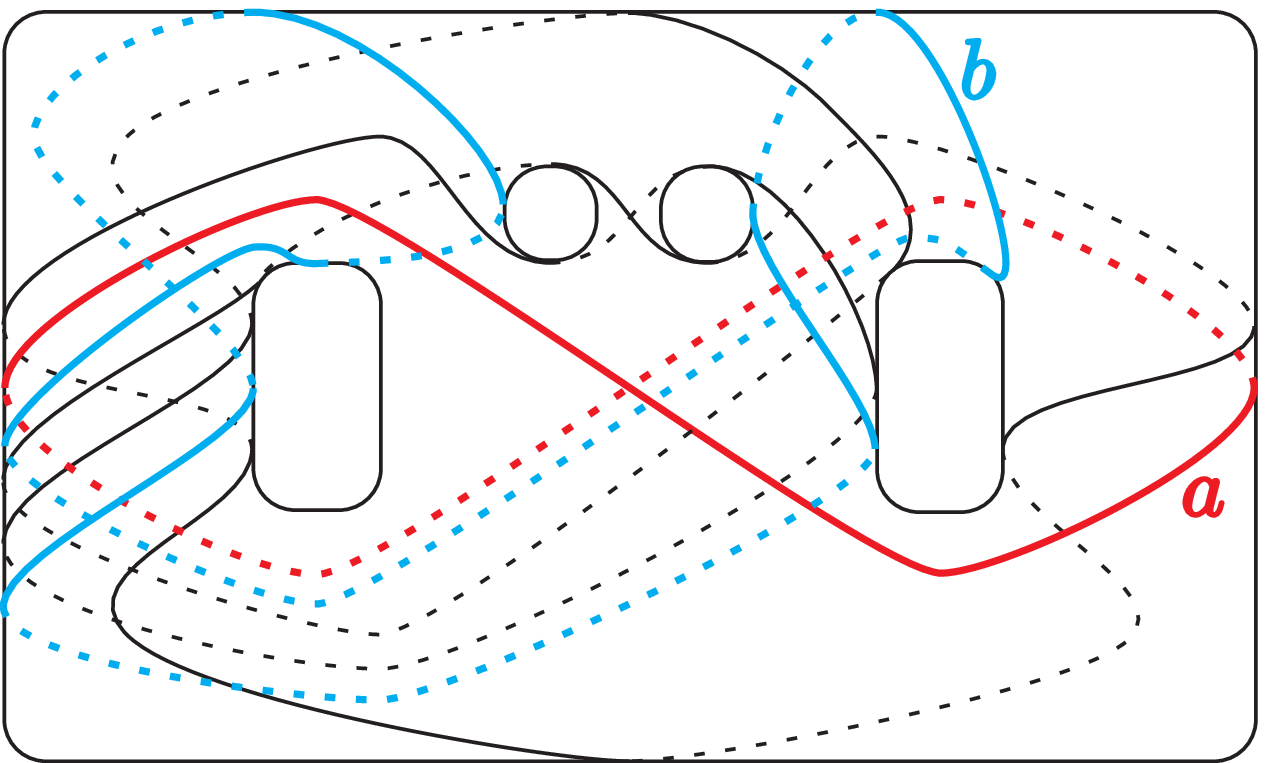}
	\end{center}
	\caption{$10_{139}$ on a genus $4$ closed surface}
	\label{10_139b}
\end{figure}

After compressing $S$ along two disks $A$ and $B$, we obtain a genus two closed surface which is a regular neighbourhood of a non-orientable, algebraically incompressible and boundary incompressible spanning surface $F'$ as in Figure \ref{10_139c}.

\begin{figure}[htbp]
	\begin{center}
	\includegraphics[trim=0mm 0mm 0mm 0mm, width=.5\linewidth]{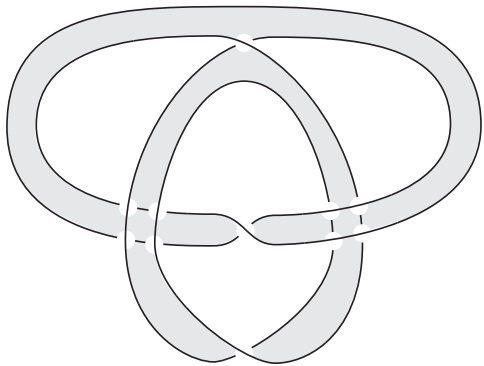}
	\end{center}
	\caption{A non-orientable, algebraically incompressible and boundary incompressible spanning surface $F'$ for $10_{139}$}
	\label{10_139c}
\end{figure}
\end{example}

\begin{theorem}
Knots with 11 crossings or less except for $K11_n118$ and $K11_n126$ satisfy the Neuwirth conjecture.
\end{theorem}

\begin{proof}
It has been confirmed in \cite{O2} that every 10 crossing knot diagram in the Rolfsen knot table \cite{R} except for these positive knots $8_{19}$, $10_{124}$, $10_{128}$, $10_{134}$, $10_{139}$ and $10_{142}$ is $\sigma$-adequate and $\sigma$-homogeneous for a positive or negative state $\sigma$ distinct from the Seifert state $\vec{\sigma}$.

$8_{19}$ and $10_{124}$ are equivalent to the $(3,4)$ and $(3,5)$-torus knot respectively and hence these knots satisfy the Neuwirth Conjecture.

$10_{128}$ and $10_{139}$ are Montesinos knots and by Theorem \ref{Montesinos}, these knots satisfy the Neuwirth Conjecture.

$10_{142}$ is the $(-4,3,3)$-pretzel knot and by Theorem \ref{pretzel}, it bounds an algebraically incompressible and boundary incompressible non-orientable pretzel surface.
Hence $10_{142}$ satisfies the Neuwirth Conjecture.

It has been be confirmed in \cite{O2} that every 11 crossing knot diagram in the Hoste-Thistlethwaite knot table \cite{HT} except for $K11_n93$, $K11_n95$, $K11_n118$, $K11_n126$, $K11_n136$, $K11_n169$, $K11_n171$, $K11_n180$ and $K11_n181$ is also $\sigma$-adequate and $\sigma$-homogeneous for a positive or negative state $\sigma$ distinct from the Seifert state $\vec{\sigma}$.

Furthermore, it can be checked that $K11_n93$, $K11_n95$, $K11_n136$, $K11_n169$, $K11_n171$, $K11_n180$ and $K11_n181$ bound algebraically incompressible and boundary incompressible non-orientable checkerboard surfaces.
(You might need to deform the diagram by the Reidemeister move of type III.)
\end{proof}

\begin{remark}
It would be interesting to prove the Neuwirth conjecture for positive knots since the $\sigma_+$-state surface for a positive state $\sigma_+$ is non-orientable.
Moreover, it seems to be not straightforward to prove the Neuwirth conjecture for positive knots, since all torus knots are positive knots.
\end{remark}

\subsection{Generalized arborescently alternating links}

In this subsection, we generalize the classes of generalized alternating links and arborescent links, and show that any link in such a huge link class bounds an algebraically incompressible and boundary incompressible generalized state surface.

First, we recall generalized alternating knots and links.
Let $F$ be a closed surface embedded in $S^3$ and $K$ a knot or link contained in $F\times [-1,1]$.
Suppose that $p(K)$ is a regular projection on $F$, where $p:F\times [-1,1]\to F\times \{0\} = F$ is the projection.
Then, we have a regular diagram on $F$ obtained from $p(K)$ by adding the over/under information to each double point, and we denote it by the same symbol $p(K)$ in this paper.
As usual, a diagram $p(K)$ on $F$ is said to be {\em alternating} if it has alternating over- and under-crossings as the diagram $p(K)$ is traversed on $F$.
We say that a diagram $p(K)$ on $F$ is {\em reduced} if there is no disk region of $F-p(K)$ which meets only one crossing.
We say that a diagram $p(K)$ on $F$ is {\em prime} if it contains at least one crossing and for any loop $l$ intersecting $p(K)$ in two points except for crossings, there exists a disk $D$ in $F$ bounded by $l$ such that $D\cap p(K)$ consists of an embedded arc.

\begin{theorem}[\cite{O}]\label{generalized alternating}
Let $F$ be a closed surface embedded in $S^3$, $K$ a knot or link contained in $F \times [-1,1]$ which has a reduced, prime, alternating diagram on $F$.
Then, both checkerboard surfaces for $K$ are algebraically incompressible and boundary incompressible.
\end{theorem}

Next, we introduce generalized arborescently alternating knots and links.
Let $S$ be a closed surface embedded in $S^3$ and $K$ a knot or link contained in $S\times [-1,1]$.
Let $p(K)$ be a connected diagram for $K$ on $S$.
A closed surface $S$ separates $S^3$ into two submanifolds, say $V_+, V_-$.
Let $\sigma : \mathcal{C}\to \{+,-\}$ be a state for $p(K)$, and let $\mathcal{C}=\{c_1,\ldots,c_n\}$ be the set of crossings of $p(K)$.
For each crossing $c_i\in \mathcal{C}$, we take a $+$-smoothing or $-$-smoothing according to $\sigma(c_i)=+$ or $-$.
See Figure \ref{resolution}.
Then, we have the set of state loops $\mathcal{L}_{\sigma} = \{l_1,\ldots, l_m\}$ on $S$.

Suppose that each state loop $l_i$ bounds two disks $d_i^+$, $d_i^-$ in $V_+$, $V_-$ respectively, and we may assume that these disks are mutually disjoint.
For each crossing $c_j$ and state loops $l_i,l_k$ whose subarcs replaced $c_j$ by $\sigma(c_j)$-smoothing, we attach a half-twisted band $b_j$ to $d_i^-$, $d_k^-$ so that it recovers $c_j$.
See Figure \ref{recover} for $\sigma(c_j)=+$.
In this way, we obtain spanning surfaces which consist of disks $d_1^-,\ldots,d_m^-$ and half-twisted bands $b_1,\ldots,b_n$ and call this {\em $\sigma$-state surfaces} $F_{\sigma}$.

\begin{remark}
Similarly to Remark \ref{option}, there are two options to choose a disk $d_i^{\pm}$ in $B_{\pm}$ bounded by each state loop $l_i$ which is not innermost in the closed surface $S$.
Therefore, in general, there are many state surfaces for a given state.
\end{remark}

We construct a graph $G_{\sigma}$ with signs on edges from $F_{\sigma}$ by regarding a disk $d_i^-$ as a vertex $v_i$ and a band $b_j$ as an edge $e_j$ which has the same sign $\sigma(c_j)$.
We call graphs $G_{\sigma}$ {\em $\sigma$-state graphs}.
Let $G_{\sigma}=G_1* \cdots *G_r$ be the block decomposition of $G_{\sigma}$.
Following \cite{LT} and \cite{C1}, we say that a diagram $p(K)$ is {\em $\sigma$-adequate} if each block $G_k$ has no loop,
 and that $p(K)$ is {\em $\sigma$-homogeneous} if in each block $G_k$, all edges have the same sign.

Furthermore, the block decomposition of $G_{\sigma}=G_1* \cdots *G_r$ corresponds to a Murasugi decomposition of $F_{\sigma}=F_1* \cdots *F_r$. See Figure \ref{decomposition_fig}.
By cutting the closed surface $S$ along each state loop $l_i$ and pasting in the disk $d_i^-$, we have closed surfaces $S_1,\ldots,S_r$ which include knot or link diagrams $p(K_1),\ldots,p(K_r)$ respectively, where $K_i=F_i$.

\begin{theorem}\label{generalized arborescently alternating}
Let $S$ be a closed surface embedded in $S^3$, $K$ a knot or link contained in $S \times [-1,1]$.
Suppose that there exists a state $\sigma$ for $p(K)$ such that:
\begin{enumerate}
\item each state loop $l_i\in \mathcal{L}_{\sigma}$ bounds two disks in both sides of $S$,
\item the diagram $p(K)$ is $\sigma$-adequate,
\item the diagram $p(K)$ is $\sigma$-homogeneous,
\item each diagram $p(K_i)$ is reduced and prime on $S_i$.
\end{enumerate}
Then, the state surface $F_{\sigma}$ for $K$ is algebraically incompressible and boundary incompressible.
Furthermore, if $\sigma$ can be taken so that it is not the Seifert state $\vec{\sigma}$, then the knot $K$ satisfies the strong Neuwirth conjecture.
\end{theorem}

\begin{proof}
The proof is similar to \cite{O2}.
By Theorem \ref{generalized alternating}, each spanning surface $F_i$ for $K_i$ is algebraically incompressible and boundary incompressible since $p(K_i)$ is a reduced, prime, alternating diagram on $S_i$.
And by Theorem \ref{Murasugi sum}, the $\sigma$-state surface $F_{\sigma}=F_1* \cdots *F_r$ is also algebraically incompressible and boundary incompressible.
\end{proof}

We call a link satisfying the conditions of Theorem \ref{generalized arborescently alternating} a {\em generalized arborescently alternating link} since it can be obtained from generalized alternating links on closed surfaces by taking Murasugi sums of those checkerboard surfaces.

\begin{remark}
The existence of knots or links which are not generalized arborescently alternating is unknown.
In \cite[Problem 1]{O2}, we have proposed a problem for showing the existence of a knot which has no $\sigma$-adequate and $\sigma$-homogeneous diagram.
\end{remark}

\subsection{Algorithms}

We show that there are algorithms to decide if a knot satisfies all of the different versions of the Neuwirth conjecture, except for the weakly strong one. The key techniques are obtained from \cite{JS} and \cite{JO}. 

\begin{theorem}\label{algorithms}
Suppose that $(S^3,K)$ is a knot. Then there are algorithms to decide if $K$ satisfies the Neuwirth conjecture, the strong Neuwirth conjecture, the even boundary slope conjecture and the strong even boundary slope conjecture.
\end{theorem}

\begin{proof}
We first consider the even and strong even boundary slope conjectures. These follow from \cite{JS} and \cite{JO} as follows. By Corollaries 3.9, 3.10, 3.11 of \cite{JS}, there is an algorithm to construct the finite set of boundary slopes of all embedded normal surfaces and hence in particular of properly embedded orientable incompressible and boundary incompressible surfaces, which can be isotoped to be normal. These occur as slopes of the finite set of vertices of the projective solution space of normal surfaces in a one vertex triangulation of $E(K)$. 

Next, by doubling $E(K)$, we obtain a closed manifold $2E(K)$ with a one vertex triangulation which is symmetric under the involution interchanging the two copies of $E(K)$. We claim that if there is an incompressible and boundary incompressible orientable surface $S$ at some slope $\alpha$ of $K$, then there is a closed orientable incompressible surface $S^*$ at a vertex of the projective solution space of $2E(K)$, where $S^* \cap E(K)$ is similarly incompressible and boundary incompressible in $E(K)$ and has the same slope $\alpha$ as $S$. Since there are finitely many vertex solutions, each can be checked to see if it has an even rational or even integer boundary slope of intersections with the torus $\partial E(K)$ and if it is incompressible, using standard techniques as discussed in either \cite{JO} or \cite{JS}. Hence the algorithm to check if a knot satisfies the even or strong even boundary slope conjectures is complete, once we have verified the claim above. 

But the proof of the claim is also straightforward. Given $S$, it is obvious that $2S$ is a closed orientable incompressible surface in $2E(K)$ meeting $\partial E(K)$ in essential curves at slope $\alpha$. As usual, we can isotope $2S$ to be least weight normal and then using \cite{JO} write a multiple of the normal class of $2S$ as a sum of vertex solutions, each of which is incompressible. So it remains to see why we can assume that these vertex solutions give at least one surface of the form $S^*$ which meets $\partial E(K)$ at slope $\alpha$. But this follows by \cite{JS}. For there, a key result is that if a properly embedded normal surface in $E(K)$ at slope $\alpha$ is written as a Haken sum, then the summands must either have the same slope $\alpha$ or a unique associated slope $\beta$. The slope $\beta$ has the property that the sum of normal curves representing $\alpha$ and $\beta$ represents a multiple of the trivial curve in a one-vertex triangulation of a torus. But in our case, it is easy to see that no associated slopes can occur since clearly our least weight normal surface in the isotopy class of $2S$ has no trivial curves of intersection with $\partial E(K)$. Hence we conclude that all the vertex solutions obtained by decomposing a multiple of $2S$ are either disjoint from $\partial E(K)$ or meet $\partial E(K)$ with slope $\alpha$ and there must be at least one surface $S^*$ of the latter type. Note that $S^* \cap E(K)$ is boundary incompressible, since it is orientable and therefore a boundary compression would imply the existence of a compression. So this completes the discussion of the even and strong even boundary slope algorithms. 

We now consider the Neuwirth conjecture and strong Neuwirth conjectures for a knot $(S^3,K)$. Suppose that there is an orientable incompressible and boundary incompressible surface $S \subset E(K)$ with $\partial S$ consisting of two curves with integer slope. The key idea is to use fundamental normal solutions rather than vertex normal solutions. Recall that a fundamental normal surface $F$ has the property that $F$ cannot be written as a Haken sum of two non-empty normal surfaces. Haken (\cite{Ha}) proved that there is a finite set of constructible fundamental normal surfaces in any closed triangulated $3$-manifold. 

The procedure is similar to the case of the boundary slope conjectures. Namely given $S$, we can assume that $2S$ is a closed orientable incompressible least weight surface in its isotopy class in $2E(K)$. Moreover we still have the key property that $2S \cap \partial E(K)$ consists of two copies of an essential curve with integer boundary slope. Now if we write $2S$ as a sum of least weight fundamental normal surfaces, using \cite{JO}, it follows that each of these is again a closed orientable incompressible surface in $2E(K)$. But we can also deduce from \cite{JS} that each of these summands either is disjoint from $\partial E(K)$ or meets $\partial E(K)$ in essential curves at the same slope as $S$. (As usual, since all these surfaces are least weight, no trivial curves of intersection with $\partial E(K)$ can occur). Moreover it is easy to see using an elementary homology argument on $\partial E(K)$ that the total number of curves of intersection of all the summands with $\partial E(K)$ must be two. Hence we see that either one summand has two such curves or two summands have one curve of intersection each. The first case gives us the result required - by checking all fundamental solutions we search for one with two intersection curves with $\partial E(K)$ which is incompressible. It then splits along $\partial E(K)$ into two surfaces each of which is a Neuwirth surface for $K$. In the second case, splitting a surface with one curve of intersection with $\partial E(K)$ gives two spanning surfaces for $K$ which must be either orientable Seifert surfaces, in which case the boundary slope of $S$ is longitudinal, or are non orientable. In the latter case, as usual, a boundary of a regular neighbourhood of such a surface is the required Neuwirth surface. Hence we have proved that there is an algorithm to verify the Neuwirth conjecture in all cases except where the boundary slope is longitudinal. 

For the final case, suppose that the Neuwirth surface $S$ we seek is longitudinal. In this case we have to complete the argument in case $2S$ splits as a sum of fundamental solutions so that two of them have one curve of intersection with $\partial E(K)$ and split along $\partial E(K)$ into orientable Seifert surfaces for $K$. It is easy to see that these two fundamental surfaces have different normal classes, so that when they are combined by a Haken sum, a connected normal surface is obtained. But then there is a simple algorithm to complete the problem. Namely search amongst fundamental normal surfaces in $2E(K)$ for ones which meet $\partial E(K)$ in one longitudinal curve. For pairs of such surfaces check whether their Haken sum is incompressible and connected. Such a surface will then split along $\partial E(K)$ into two longitudinal Neuwirth surfaces as required. 

It remains to discuss the strong Neuwirth conjecture. This is very similar to the argument for the Neuwirth conjecture. Assume that $S$ is a non-orientable surface at an even integer boundary slope in $E(K)$, so that $S^\prime=\partial N(S)$ is incompressible and boundary incompressible in $E(K)$. By \cite{Ha}, we can write $[2S]=\Sigma n_i[F_i ]$ where $F_1,F_2, \ldots, F_k$ are fundamental normal surfaces in $2E(K)$ and $n_i$ are positive integers. By the arguments in the previous paragraphs, it follows that there must be some $i_0$ so that $n_{i_0}=1$ and $F_{i_0}$ meets $\partial E(K)$ in a single curve at the same boundary slope as $S$. Moreover each $F_i$ is disjoint from $\partial E(K)$ for each $i \ne i_0$. (This follows since the results of \cite{JS} on boundary slopes apply to normal surfaces, not just incompressible normal surfaces). We see that if $F_{i_0}$ is split along $\partial E(K)$, the result is either non-orientable spanning surfaces at the same boundary slope as $S$ or Seifert surfaces, if the slope of $S$ is longitudinal. However if all the surfaces $F_i$ are orientable, then so is $[2S]$, which is a contradiction. On the other hand, each $F_i$, for $i \ne i_0$, is a closed surface in $E(K)$ so must be orientable. Hence we conclude that $F_{i_0}$ must be non-orientable and must split along $\partial E(K)$ into at least one non-orientable spanning surface for $K$.

To complete the argument, we want to apply \cite{JO}. The equation $[2S]=\Sigma n_i[F_i ]$ implies $[2S^\prime] =2[2S]=\Sigma 2n_i[F_i ]$. But then by \cite{JO}, we deduce that all the surfaces $F_i$ are $\pi_1$-injective and so in particular $F_{i_0}$ is $\pi_1$-injective. Therefore $\partial N(F_{i_0})$ meets $E(K)$ in an incompressible and boundary incompressible surface. So this shows that $F_{i_0} \cap E(K)$ is an essential non-orientable spanning surface for $K$. We can now verify the strong Neuwirth conjecture, by searching amongst fundamental normal surfaces in $2E(K)$ for a solution $F_{i_0}$ with the property that $\partial N(F_{i_0})$  is incompressible in $2E(K)$ and $F_{i_0} \cap \partial E(K)$ has even integer boundary slope. 
\end{proof}

\bigskip

\noindent{\bf Acknowledgement.}
We would like to thank Kazuhiro Ichihara for giving us a helpful suggestion about Lemma \ref{deform}.
The research and writing of this work was carried out while the first author was visiting at Department of Mathematics and Statistics, The University of Melbourne during April 2009--March 2011.
The first author would like to thank the second author for his hearty welcome and a lot of joint work.

\bibliographystyle{amsplain}

\end{document}